\documentclass[a4paper,twoside,10pt]{article}

\usepackage[a4paper,left=3cm,right=3cm, top=3cm, bottom=3cm]{geometry}
\usepackage[latin1]{inputenc}
\usepackage{mathrsfs}
\usepackage{graphicx}
\usepackage{epstopdf}
\usepackage{amsmath}
\usepackage{amsthm}
\usepackage{amssymb}
\usepackage{hyperref}
\usepackage{stmaryrd}
\usepackage{float}
\usepackage{bigints}
\usepackage{cite}
\usepackage{color}
\usepackage[abs]{overpic}
\usepackage[font=footnotesize,labelfont=bf]{caption}
\usepackage{cases}
\usepackage{tikz}
\usepackage{algorithmic}
\usepackage{rotating}
\usepackage{blkarray}
\usetikzlibrary{matrix,calc,arrows}
\usepackage[author={Lorenzo}]{pdfcomment}
\usepackage{soul,xcolor}
\usepackage{verbatim}
\usepackage{graphicx}
\usepackage{subcaption}
\usepackage{algorithm}
\usepackage{pifont}
\usepackage{multirow}

\restylefloat{table}
\theoremstyle{plain}
\newtheorem{thm}{Theorem}[section]
\newtheorem{cor}[thm]{Corollary}
\newtheorem{lem}[thm]{Lemma}

\theoremstyle{definition}

\theoremstyle{remark}
\newtheorem{remark}{Remark}

\newcommand{\red}[1]{{\color{red}{#1}}} 
\newcommand{\yellow}[1]{{\color{yellow}{#1}}} 
\newcommand{\green}[1]{{\color{green}{#1}}} 

\newcommand{\h}{h}
\newcommand{\p}{p}
\newcommand{\pbold}{\mathbf \p}
\newcommand{\pboldEpsilon}{\pbold_{\mathcal E}}
\newcommand{\K}{\mathbb K} 
\renewcommand{\div}{\text{div}}
\newcommand{\E}{K}
\newcommand{\SEo}{S^\E_1}
\newcommand{\SEz}{S^\E_0}
\newcommand{\un}{u_n}
\newcommand{\vn}{v_n}
\newcommand{\Pinabla}{\Pi^\nabla_\p}
\newcommand{\Pinablapmo}{\Pi^\nabla_{\p-1}}
\newcommand{\Piz}{\Pi^0_{\p-1}}
\newcommand{\Pizpmo}{\Pi^0_{\p-1}}
\newcommand{\Pizpmt}{\Pi^0_{\p-2}}
\newcommand{\hE}{\h_\E}
\newcommand{\he}{\h_\e}
\newcommand{\n}{\mathbf n}
\newcommand{\phin}{\utildepi}
\newcommand{\upi}{u_\pi}
\newcommand{\utildepi}{\widetilde u_\pi}
\newcommand{\uI}{u_I}
\newcommand{\vI}{v_I}
\newcommand{\taun}{\mathcal T_n}
\newcommand{\En}{\mathcal E_n}
\newcommand{\Enb}{\mathcal E_n^B}
\newcommand{\EnI}{\mathcal E_n^I}
\newcommand{\tautilden}{\widetilde{\mathcal T}_n}
\newcommand{\Vn}{V_n}
\newcommand{\VnE}{\Vn(\E)}
\newcommand{\VtildenE}{\widetilde{V}_n(\E)}
\newcommand{\EE}{\mathcal E^\E}
\newcommand{\e}{e}

\newcommand{\malpha}{m_{\alpha}}
\newcommand{\dof}{\text{dof}}
\renewcommand{\a}{a}
\renewcommand{\b}{b}
\renewcommand{\c}{c}
\newcommand{\aE}{\a^\E}
\newcommand{\bE}{\b^\E}

\newcommand{\an}{\a_n}
\newcommand{\bn}{\b_n}
\newcommand{\cn}{\c_n}
\newcommand{\anE}{\an^\E}
\newcommand{\bnE}{\bn^\E}
\newcommand{\cnE}{\cn^\E}
\newcommand{\V}{V}
\newcommand{\f}{f}

\newcommand{\lambdan}{\lambda_n}
\newcommand{\deltan}{\delta_n}
\newcommand{\T}{T}
\newcommand{\Tn}{\T_n}
\newcommand{\B}{\mathcal B}
\newcommand{\Bn}{\B_n}
\newcommand{\kcal}{\mathfrak{k}}
\newcommand{\vbold}{\mathbf v}
\newcommand{\wn}{w_n}
\newcommand{\M}{\cal M}
\newcommand{\card}{\text{card}}

\newcommand{\qpmo}{q_{\p-1}}
\newcommand{\nE}{n^\E}
\newcommand{\Omegatilde}{\widetilde \Omega}

\begin{tiny}
\author{
\normalsize{
}}
\end{tiny}

\date{}

\title{\small{\textbf{The $\p$- and $\h\p$-versions of the virtual element method for elliptic eigenvalue problems}}}
\date{}
\author{\footnotesize{O. \v{C}ert\'ik \thanks{Group CCS-2, Computer, Computational and Statistical Division, Los Alamos National Laboratory, 87545 Los Alamos, New Mexico, USA (certik@lanl.gov)},\
F. Gardini\thanks{Dipartimento di Matematica F. Casorati, Universit\`a di Pavia, 27100 Pavia, Italy (francesca.gardini@unipv.it)},\
G. Manzini\thanks{Group T-5, Theoretical Division, Los Alamos National Laboratory, 87545 Los Alamos, New Mexico, USA (gmanzini@lanl.gov)},\ 
L. Mascotto\thanks{Fakult\"at f\"ur Mathematik, Universit\"at Wien, 1090 Vienna, Austria (lorenzo.mascotto@univie.ac.at)},\ 
G. Vacca\thanks{Dipartimento di Matematica e Applicazioni, Universit\`a degli Studi di Milano Bicocca, 20125 Milano, Italy (giuseppe.vacca@unimib.it)}}}

\begin{document}
\maketitle
\begin{abstract}
\noindent
We discuss the $\p$- and the $\h\p$-versions of the virtual element method for the approximation of eigenpairs of elliptic operators with a potential term on polygonal meshes.
An application of this model is provided by the Schr\"odinger equation with a pseudo-potential term.
We present in details the analysis of the $\p$-version of the method, proving exponential convergence in the case of analytic eigenfunctions. The theoretical results are supplied with a wide set of experiments.
We also show numerically that, in the case of eigenfunctions with finite Sobolev regularity, an exponential approximation of the eigenvalues in terms of the cubic root of the number of degrees of freedom can be obtained by employing $\h\p$-refinements.
Importantly, the geometric flexibility of polygonal meshes is exploited in the construction of the $\h\p$-spaces.



\medskip\noindent
\textbf{AMS subject classification}: 65L15, 65N15, 65N30

\medskip\noindent
\textbf{Keywords}: virtual element methods, polygonal meshes, eigenvalue problems, $\p$- and $\h\p$-Galerkin methods
\end{abstract}

\section{Introduction} 
In the last five years, the virtual element method (VEM)~\cite{VEMvolley,equivalentprojectorsforVEM}, has established itself as one of the most ductile and flexible Galerkin methods for the approximation of solutions to partial differential equations (PDEs) on polygonal and polyhedral meshes, i.e., meshes with arbitrarily-shaped polygonal/polyhedral (polytopal, for short) elements.
Implementation details can be found in~\cite{hitchhikersguideVEM}.
The method has been proved to be very successful for a number of mathematical/engineering problems,
an extremely short list being given by References~\cite{vacca2018h,VEMmagnetostatic2D,fumagalli2018sisc,artioli2017CMAME,VEM_fullync_biharmonic,absv_VEM_cahnhilliard,Wriggers-contact,VEM_DD_basic,gain2014virtual}.

The VEM is a generalization of the finite element method (FEM) to polygonal grids~\cite{Manzini-Russo-Sukumar:2014}, and is based on tools stemming from the mimetic finite differences~\cite{BLM_MFD, lipnikov2014mimetic}.
The VEM is a generalization of the finite element method (FEM) to polygonal grids, and is based on tools stemming from the mimetic finite differences~\cite{BLM_MFD, lipnikov2014mimetic}.
The main idea of the method is that, to standard piecewise polynomials, additional functions  allowing the construction of suitable global space are added;
such functions are defined implicitly as solutions to local PDEs and therefore are unknown in closed form.
As a consequence, the exact forms appearing in the weak formulation of the problem are not computable;
rather, they are replaced by suitable discrete counterparts that have to be computable in terms of the degrees of freedom and that are based on two main ingredients:
projectors onto polynomial spaces, and bilinear forms stabilizing the method on the kernel of such projectors.

The aim of the present work is to discuss the approximation of eigenpairs of certain elliptic operators by means of VEM.
Despite the novelty of this method, virtual elements for the approximation of eigenvalues have been applied to a plethora of different problems, such as
the Poisson problem~\cite{gardini2018nonconforming,VEM_eig_basic}, the Poisson problem with a potential term~\cite{VEM_eig_Schroedinger}, the Steklov eigenvalue problem~\cite{MoraRiveraRodriguez_aposSteklov, VEMchileans},
transmission problems~\cite{trasmissionVEMeig}, the vibration problem of Kirchhoff plates~\cite{VEM_vibration_Kirchhoff}, and the acoustic vibration problem~\cite{acousticVEM}.
We also highlight that the approximation of eigenvalues with polygonal methods has been targeted in the context of the hybrid-high order method~\cite{HHO2017eig} and of the mimetic finite differences~\cite{MFD2011eig}.

In all the above-mentioned approaches, the focus of the analysis is the so-called $\h$-version of the method, i.e., the convergence of the error is achieved by keeping fixed the dimension of the local spaces, and by refining the underlying polygonal grids.
One of the novelty of the present paper is that we investigate the approximation by means of VEM of eigenvalue problems employing both the $\p$- and the $\h\p$-versions of the method.
In the former approach, the convergence is obtained by keeping fixed the mesh and by increasing the dimension of the local spaces.
The latter approach, see~\cite{SchwabpandhpFEM, BabuGuo_hpFEM}, makes instead use of a combination of the $\h$- and of the $\p$-versions;
in particular, the $\h\p$-gospel states that the meshes have to be refined on those elements where the exact solution has a finite Sobolev regularity,
whereas the polynomial degree increases in a nonuniform fashion on those elements where the solution is smooth.

The advantage of using the $\p$- and the $\h\p$-versions of a Galerkin method over their $\h$-counterpart, is that, in the latter case, the method converges algebraically in terms of the mesh size,
with rate depending on the polynomial degree and on the regularity of the solution.
On the contrary, exponential convergence can be proven in the former cases; more precisely, for analytic solutions, the $\p$-version converges exponentially in terms of the polynomial degree~$\p$,
whereas, for solutions with finite Sobolev regularity, the $\h\p$-version gives exponential convergence in terms of the cubic root of the number of degrees of freedom.
The literature of $\p$- and $\h\p$-continuous and discontinuous FEM for the approximation of the eigenvalues is particularly wide.
We limit ourselves here to cite the works of Giani and collaborators, see for instance~\cite{giani2013errorhp, giani2014hp, giani2015hp} and
a paper of Sauter~\cite{sauter2010hp}, where error estimates are proven with bounds that are explicit in the mesh size, in the polynomial degree, and in the eigenvalues;
the work~\cite{davydov2017convergence} focuses instead on $\h\p$-adaptive FEM in the framework of eigenvalues in quantum mechanics.

The $\p$- and the $\h\p$-versions of VEM have been investigated in a series of works:
the analysis for quasi-uniform and geometrically graded meshes was the topic of~\cite{hpVEMbasic, hpVEMcorner},
a $\p$-multigrid algorithm was investigated in~\cite{pVEMmultigrid}; finally, \cite{hpVEMapos} was devoted to $\h\p$-residual-based a posteriori error analysis.
In all these works, the target problem was the Poisson problem.

An additional novelty of this paper is that we extend the $\p$- and $\h\p$-analysis of VEM to the case of more general elliptic problems,
namely, we allow for variable diffusivity tensor and for the presence of a (smooth) potential term.
With respect to the Poisson case, we face here additional hindrances due to the fact that we employ some special virtual element spaces, that is, the so-called enhanced virtual element spaces~\cite{equivalentprojectorsforVEM}:
(i) $\p$-best interpolation estimates in enhanced virtual element spaces can be suboptimal; (ii) a stabilization for the $L^2$ inner product with explicit bounds in terms of~$\p$ has to be figured out.
Moreover, at the practical level, one has to be careful in defining a ``clever'' basis of the space, since a bad choice could lead to a very ill-conditioned method;
in order to avoid such situation, we will resort to the special bases discussed in~\cite{fetishVEM, fetishVEM3D}.

As already underlined, we focus on the approximation of the eigenvalues and eigenfunctions of elliptic operators consisting of a second order term (with variable diffusion tensor) plus a zero-th order pseudo-potential term.
This corresponds to the case of a Schr\"odinger equation with a pseudo-potential term, which is a basic brick to face more complex problems stemming from the density functional theory~\cite{bader1991quantum, gross2013density, yang2003density}.
We highlight that the analysis for more general elliptic problems, e.g. including a convective term, follows combining the techniques of the present paper with those in~\cite{bbmr_VEM_generalsecondorderelliptic}.

The paper is organized as follows. Having introduced the method (including the local and global  discrete spaces, and the discrete bilinear forms)
and its approximation properties in Section~\ref{section:the:virtual:element:method},
we discuss the convergence analysis in Section~\ref{section:convergence:analysis}; here, we use the tools stemming from the Babu\v ska-Osborn theory~\cite{BabuskaOsborn}, see also~\cite{boffiAN}.
Section~\ref{section:numerical:results} is committed to present a number of numerical experiments including the $\p$- and the $\h\p$-versions of the method;
in the latter case, we employ meshes that geometrically graded towards the singularities of the eigenfunctions;
the construction of such graded meshes exploits the geometric flexibility of polygonal meshes.
The conclusions are stated in Section~\ref{section:conclusion}.

\paragraph*{Notation}
Throughout the paper, we shall employ the standard notation for Sobolev spaces. In particular, given $D \subseteq \mathbb R^n$, $n=1,2$, we denote by $H^s(D)$, $s\in \mathbb R_+$, the Sobolev space of order~$s$ over~$D$ and we denote by
\[
(\cdot, \cdot)_{s,D},\quad\quad \vert \cdot \vert_{s,D}, \quad\quad \Vert \cdot \Vert_{s,D},
\]
the associated $H^s$ inner product, seminorm, and norm, respectively.

The space $H^{\frac{1}{2}}(\partial D)$, where $\partial D$ is a Lipschitz boundary,
is defined as the space of $H^0(\partial D) = L^2(\partial D)$ functions over~$\partial D$, with finite A\-ron\-szajn-Slo\-bo\-de\-ckij seminorm
\[
\vert u \vert _{\frac{1}{2}, \partial D}^2 = \int_{\partial D}\int_{\partial D} \frac{\vert u(\xi) - u(\eta)\vert^2}{\vert \xi - \eta \vert^2} d\xi\,d\eta.
\]
The space $H^{-\frac{1}{2}}(\partial D)$ represents the dual space of $H^{\frac{1}{2}}(\partial D)$.
Besides, we denote by $\mathbb P_\ell(D)$, $\ell \in \mathbb N$, the space of polynomials of degree smaller than or equal to~$\ell$ over~$D$, and by~$\pi_\ell$ its dimension.
Instead, given~$\ell_1$ and~$\ell_2$ such that~$\ell_1< \ell_2$, the space $\mathbb P_{\ell_2} (D)\setminus \mathbb P_{\ell_1} (D)$ represents the space of the polynomials of degree~$\ell_2$, that are orthogonal in~$L^2(D)$ to the space~$\mathbb P_{\ell_1}(D)$.
Finally, given two positive quantities~$a$ and~$b$, we write~$a \lesssim b$ in lieu of ``there exists a constant $c$, independent of the mesh size and of the polynomial degree, such that $a \le c\, b$''.
Moreover, we write~$a\approx b$ meaning~$a \lesssim b$ and~$b\lesssim a$ at the same time.


\paragraph*{The continuous problem}
Given a domain~$\Omega \subset \mathbb R^2$,
let $\K: \Omega \rightarrow \mathbb R^{2\times 2}$ be a symmetric positive definite tensor with
\begin{equation} \label{property_K}
\kcal_* \vert \vbold \vert_{\ell^2} \le  (\K \vbold) \cdot \vbold \le \kcal^* \vert \vbold \vert_{\ell^2} \quad \forall \vbold \in \mathbb R^2,
\end{equation}
where $\vert \cdot \vert_{\ell^2}$ denotes the Euclidean norm in~$\mathbb R^2$, for some given constants~$0<\kcal_*\le \kcal^*$ independent of the discretization parameters,
and let $\V: \Omega \rightarrow \mathbb R$ be such that
\begin{equation} \label{property_V}
0\le \nu_* \le \V \le \nu^* \quad \text{almost everywhere in~$\Omega$},
\end{equation}
for some given constants~$\nu_*$ and $\nu^*$ independent of the discretization parameters.

We look for nonzero functions~$u$ and for positive real numbers~$\lambda$ satisfying
\begin{equation} \label{continuous:problem:strong}
\begin{cases}
-\div(\K \cdot \nabla u) + \V \, u = \lambda u & \text{in }\Omega\\
u=0 & \text{on }\partial \Omega.
\end{cases}
\end{equation}
Note that~$u$ is defined up to a multiplicative factor. We decide to fix~$\Vert u \Vert_{0,\Omega}=1$.

In weak formulation, the eigenvalue problem~\eqref{continuous:problem:strong} reads:
\begin{equation} \label{continuous:problem:weak}
\begin{cases}
\text{find }(\lambda,u) \in \mathbb R \times H^1_0(\Omega) \text{ with $\Vert u \Vert_{0,\Omega}=1$ such that}\\
\a(u,v) + \b(u,v)= \lambda \c(u,v) \quad \forall v \in H^1_0(\Omega),\\
\end{cases}
\end{equation}
where we have set
\begin{equation} \label{notation:for:bilinear:forms}
\a(u,v) = (\K \nabla u, \nabla v)_{0,\Omega},\quad \b(u,v) = (\V  u, v)_{0,\Omega},\quad \c(u,v) = (u, v)_{0,\Omega} \quad \forall u,\, v \in H^1(\Omega).
\end{equation}
We will also make use of the source problem associated with~\eqref{continuous:problem:weak}: given $\f\in H^1(\Omega)$,
\begin{equation} \label{weak:formulation:continuous:source}
\begin{cases}
\text{find } u \in H^1_0(\Omega) \text{ such that}\\
\a(u,v) + \b(u,v)= \c(\f,v) \quad \forall v \in H^1_0(\Omega).\\
\end{cases}
\end{equation}
We define the solution operator $\T \in \mathcal L(H^1(\Omega))$ of the source problem~\eqref{weak:formulation:continuous:source} as
\begin{equation*}
\B(\T \f,v)  = \c(\f,v)\quad \forall v\in H^1_0(\Omega),
\end{equation*}
where we have set $\B(\cdot, \cdot) = \a(\cdot, \cdot) + \b(\cdot, \cdot)$. The operator~$\T$ is self-adjoint, compact (thanks to the Sobolev embedding theorems), and positive definite.

We observe that, given $(\lambda,w)\in \mathbb R \times H^1_0(\Omega)$ an eigenpair solution to~\eqref{continuous:problem:weak}, then
\[
\a(w,v) + \b(w,v) = \c(\lambda w,v)  = a(\T \lambda w,v) + b(\T \lambda w,v)  \quad \forall w  \in H^1_0(\Omega).
\]
Thus, $\T (\lambda w) = w$ and then~$\T (w) =\frac{1}{\lambda} w$, which means that $(\frac{1}{\lambda}, w)$ is an eigenpair of~$\T$.
As a consequence, in order to approximate the eigenfunctions and the eigenvalues of the problem~\eqref{continuous:problem:weak}, it suffices to approximate the eigenfunctions and eigenvalues of~$\T$.
This will be tackled with the tools provided by the Babu\v ska-Osborn theory~\cite{BabuskaOsborn}.

\section{The virtual element method} \label{section:the:virtual:element:method}
In this section, we discuss the virtual element method tailored for the approximation of the solutions to problem~\eqref{continuous:problem:weak}, and we discuss $\h$- and $\p$-approximation properties of the functions in such spaces.

More precisely, after having introduced the concept of regular polygonal decompositions in Section~\ref{subsection:regular:polygonal:meshes}, in Section~\ref{subsection:virtual:element:spaces} we construct the virtual element spaces
and we describe their $\h$- and $\p$-approximation properties; the definition of the discrete bilinear forms is instead the topic of Section~\ref{subsection:discret:bilinear:forms},
where a particular emphasis is put on the analysis of the stabilizations with explicit bounds in terms of the mesh size~$\h$ and the ``polynomial'' degree~$\p$.
Finally, Section~\ref{subsection:the:method} is devoted to state the method.

\subsection{Polygonal meshes} \label{subsection:regular:polygonal:meshes}
Given $\Omega \subset \mathbb R^2$, we introduce here the concept of regular polygonal decompositions and some useful notation, instrumental for the description of the method.

Let~$\{ \taun\}_{n\in \mathbb N}$ be a sequence of conforming polygonal decompositions of~$\Omega$, i.e., for all $n \in \mathbb N$, $\taun$ is a collection of polygons,
such that the intersection of two different polygons is either the empty set, a vertex, or a collection of edges.

We set for future convenience~$\En$, $\Enb$, and~$\EnI$ the set of edges, of boundary edges, and internal edges of~$\taun$, respectively;
moreover, we set~$\EE$ the set of edges of $\E$ and~$\nE$ its cardinality.
The diameter of the elements~$\E \in \taun$, the mesh size of~$\taun$, and the length of the edges~$\e\in \En$, are denoted by
\[
\hE = \text{diam}(\E),\quad  \h = \max_{\E \in \taun} \hE,\quad \he = \text{length} (\e),
\]
respectively.

In the forthcoming analysis, we will make use of the two following assumptions.
For all $n \in \mathbb N$, there exists a positive constant~$\gamma$ independent of the mesh size such that
\begin{itemize}
\item [(\textbf{D1})] for all $\E\in \taun$ and for all $\e \in \EE$, it holds that~$\he$ is larger than or equal to~$\gamma \hE$;
\item [(\textbf{D2})] every $\E \in \taun$ is star-shaped with respect to at least one ball, with radius larger than or equal to~$\gamma \hE$.
\end{itemize}

We underline that the assumptions (\textbf{D1}) and (\textbf{D2})
could be in principle weakened, yet retaining analogous approximation properties of the method, as discussed in~\cite{beiraolovadinarusso_stabilityVEM, chen_anisotropic_conforming, brennerVEMsmall}.
\medskip

Given~$\taun$ a polygonal decomposition of~$\Omega$, we define the broken Sobolev seminorm
\begin{equation}\label{broken Sobolev}
\vert \cdot \vert^2_{1,\taun} = \sum_{\E \in \taun} \vert \cdot \vert^2_{1,\E}.
\end{equation}
Furthermore, we introduce two families of operators
\begin{align}
&\Pi_\ell^{0, \Omegatilde} : L^2(\Omegatilde) \rightarrow \mathbb P_\ell(\Omegatilde), \quad (\Pi_\ell^{0, \Omegatilde} u - u, q_\ell)_{0,\Omegatilde}=0 \quad \forall u \in L^2(\Omegatilde),\; \forall q_\ell \in \mathbb P_\ell(\Omegatilde),    \label{full:projector:a} \\
& \Pi_\ell^{\nabla,\Omegatilde} : H^1(\Omegatilde) \rightarrow \mathbb P_{\ell}(\Omegatilde), \quad
\begin{cases}
(\nabla \Pi_\ell^{\nabla,\Omegatilde} u - \nabla u, \nabla q_\ell)_{0,\Omegatilde}=0 \\
\int_{\partial \Omegatilde} u - \Pi_\ell^{\nabla,\Omegatilde} u = 0\\
\end{cases}
\quad \forall u \in H^1(\Omegatilde),  \;\forall q_\ell \in \mathbb P_\ell(\Omegatilde),    \label{full:projector:c} 
\end{align}
for some measurable set~$\Omegatilde\subseteq \mathbb R^2$ and for some~$\ell \in \mathbb N$.
In the following, we will denote by~$\Pi_\ell^{0, \Omegatilde}$ also the vector version of the~$L^2$ projector defined in~\eqref{full:projector:a}.
\medskip

We also consider the following simplifying assumption:
\begin{itemize}
\item [(\textbf{A})] the coefficients~$\K$ and~$\V$ in~\eqref{continuous:problem:strong} are piecewise analytic over~$\taun$, for all $n\in \mathbb N$.
\end{itemize}

\subsection{Virtual element spaces} \label{subsection:virtual:element:spaces}
In the present section, we introduce the local and global virtual element spaces for the problem~\eqref{continuous:problem:weak} and
we study their~$\h$- and $\p$-approximation properties.

Given an element $\E \in \taun$ and the ``polynomial'' degree~$\p\in \mathbb N$, we define the auxiliary space
\begin{equation*}
\VtildenE = \{  \vn \in \mathcal C^0(\overline\E) \mid \Delta \vn \in \mathbb P_{\p-1} (\E),\; \vn{}_{|\e} \in \mathbb P_{\p}(\e) \, \forall \e\in \EE  \}.
\end{equation*}
The local virtual element space over the element~$\E$ reads
\begin{equation} \label{local:enhanced:space}
\VnE = \left\{ \vn \in \VtildenE \mid \int_\E (\vn - \Pi^{\nabla, \E}_\p \vn) \malpha = 0 \quad \forall \malpha \in \mathbb P_ {\p-1}(\E) \setminus \mathbb P_{\p-2}(\E) \right\}.
\end{equation}
The local space~$\VnE$ has been constructed in the spirit of the enhanced virtual element space, see~\cite{equivalentprojectorsforVEM}.
It is essential to underline that the local virtual element space~$\VnE$ contains the space of polynomials of degree smaller than or equal to~$\p$.
This fact guarantees that $\VnE$ has good approximation properties.

We recall from~\cite{VEMvolley, equivalentprojectorsforVEM}, that $\VnE$can be endowed with the following set of unisolvent degrees of freedom.
Given~$\vn \in \VnE$,
\begin{itemize}
\item for all the vertices~$\{\nu_i\}_{i=1}^{\nE}$ of~$\E$, the point-values~$\vn(\nu_i)$, for all $i=1,\dots, \nE$;
\item for all edges~$\e \in \EE$, the point-values at~$\p-1$ distinct internal points of~$\e$ (e.g. at the~$\p-1$ internal Gau\ss-Lobatto nodes);
\item given $\{\malpha\}_{\alpha=1}^{\pi_{\p-2}}$, where we recall that~$\pi_{\p-2}$ denotes the dimension of~$\mathbb P_{\p-2}(\E)$, \emph{any} basis of~$\mathbb P_{\p-2}(\E)$ invariant under homothetic transformation\, the scaled moments
\begin{equation} \label{internal:moments}
\frac{1}{\vert \E \vert} \int_\E \vn \malpha.
\end{equation}
\end{itemize}
In the original VEM approach~\cite{VEMvolley}, as well as in the majority of the literature, the basis $\{\malpha\}_{\alpha=1}^{\pi_{\p-2}}$ is chosen as the basis of (scaled and shifted with the barycenter of the element) monomials.
However, in presence of very distorted elements or for a high ``polynomial'' degree~$\p$, this choice may result in a loss of accuracy in the method;
alternative choices tackling the high-order case are available in literature~\cite{fetishVEM, fetishVEM3D}, and will be pinpointed in Section~\ref{section:numerical:results}.

Having at disposal the set of local unisolvent degrees of freedom $\{ \dof_j \}_{j=1}^{\dim(\VnE)}$, we introduce the canonical basis $\{\dof_j\}_{j=1}^{\dim(\VnE)}$ as $\dof_j(\varphi_\ell) = \delta_{j,\ell}$, where~$\delta_{j,\ell}$ denotes the Kronecker delta.

Importantly, the functions in local virtual element spaces, as they are solutions to local Poisson problems, are known explicitly only at the boundary of the element, but are unknown in closed form at the interior.
Therefore, the exact bilinear forms are not computable; rather, they have to be replaced by proper discrete counterparts avoiding the evaluation of trial and test functions at the integration points, see Section~\ref{section:convergence:analysis}.

We also highlight that the choice of the degrees of freedom allows to compute \emph{explicitly} the following quantities, see e.g.~\cite{bbmr_VEM_generalsecondorderelliptic}:
\[
\Pi^{0,\E}_{\p-1} \un, \quad \quad \Pi^{\nabla,\E}_\p \un,\quad \quad \Pi^{0,\E}_{\p-1} (\nabla \un) \quad \quad \forall \un \in \VnE,
\]
where the projectors~$\Pi^{0,\E}_{\p-1}$ and~$\Pi^{\nabla,\E}_\p$ are defined in~\eqref{full:projector:a} and~\eqref{full:projector:c}, respectively, and where, with a slight abuse of notation, $\Pi^{0,\E}_{\p-1} (\nabla \un)$ is the vector counterpart
of the~$L^2$ projector in~\eqref{full:projector:a}.
For the sake of clarity, we will drop the superscript~$\E$.
\medskip

The global space is built in an $H^1$-conforming fashion:
\[
\Vn = \{  \vn \in C^0(\overline \Omega) \cap H^1_0 (\Omega) \mid \vn{}_{|\E} \in \VnE \text{ for all } \E \in \taun     \}.
\]
The set of global degrees of freedom is obtained by a standard coupling of the local ones.

We underline that it is also possible to build nonconforming spaces (\`a la Crouzeix-Raviart), see e.g.~\cite{nonconformingVEMbasic, gardini2018nonconforming}.
At any rate, we stick here to the~$H^1$-conforming case.

The remainder of the section is devoted to prove some best approximation results in polynomial and virtual element spaces.
We begin by recalling from~\cite[Lemma 4.2]{hpVEMbasic} the following $\h\p$-best polynomial approximation results over shape regular polygons.
\begin{thm}[$\h\p$-best polynomial approximation error over polygons] \label{theorem:best:polynomial:error}
Given $\E\in\taun$ and $u \in H^{s+1}(\E)$, $s\ge0$, for all $\p \in \mathbb N$, there exists $\upi \in \mathbb P_\p(\E)$ such that
\begin{equation} \label{cazzo}
\vert u - \upi \vert_{\ell,\E} \lesssim  \frac{\hE^{\min(\p,s)+1-\ell}}{\p^{s+1-\ell}} \Vert u \Vert_{s+1,\E} \quad \forall \ell \ge 0 \quad \text{such that} \quad 0\le \ell \le s.
\end{equation}
\end{thm}
Next, we prove an auxiliary result which will be instrumental for proving an $\h\p$-best interpolation result in enhanced virtual element spaces.
To this aim, we first introduce~$\tautilden$, a subtriangulation of $\Omega$ obtained as follows.
For every $\E\in\taun$, we connect its vertices to the center of any ball (with maximal radius) with respect to which~$\E$ is star-shaped, see the geometric assumption (\textbf{D2}).
The union of such triangles is denoted by~$\tautilden$.
Associated with the subtriangulation~$\tautilden$, we define the broken Sobolev seminorm $\vert \cdot \vert_{1,\tautilden}^2$ as in~\eqref{broken Sobolev}.
\begin{thm} \label{theorem:best:interpolation:error:auxiliary}
For every $u\in H^1(\Omega)$, there exists $\uI\in \Vn$ such that
\begin{equation} \label{best:interpolation:error}
\vert u - \uI \vert_{1,\Omega} \lesssim \p (\vert u - \upi \vert_{1,\taun} + \vert u - \utildepi \vert_{1,\tautilden})
\end{equation}
for all~$\upi$ and~$\phin$ piecewise continuous polynomials of degree $\p$ over $\taun$ and $\tautilden$, respectively.
\end{thm}
\begin{proof}
The proof employs some tools from~\cite[Lemma 4.3]{hpVEMbasic}, ~\cite[Theorem 2]{hpVEMcorner}, and~\cite[Theorem 11]{cangianigeorgulispryersutton_VEMaposteriori}.
We assume without loss of generality that~$\p > 2$, since we are interested only in the asymptotic behavior of~$\p$.

Given~$u \in H^1(\Omega)$, we start by defining the auxiliary interpolant~$\vI\in H^1(\Omega)$ whose restriction on~$\E$, for all~$\E \in \taun$, belongs to the space~$\VtildenE$, as
\[
\begin{cases}
-\Delta \vI = -\Delta \upi & \text{in }\E\\
\vI=\phin & \text{on } \partial \E,\\
\end{cases}
\]
for some~$\upi$ and~$\phin$ piecewise continuous polynomials of degree~$\p$ over~$\taun$ and~$\tautilden$, respectively.

Following~\cite[Lemma 4.3]{hpVEMbasic} or \cite[Proposition 4.2]{VEMchileans}, one shows
\begin{equation}\label{Chileans:estimate}
\vert u - \vI \vert_{1,\Omega} \le \vert u - \phin \vert_{1,\tautilden} + 2\vert u - \upi \vert_{1,\taun}.
\end{equation}
Next, we introduce an interpolant~$\uI$ in the space~$\Vn$, defined as
\begin{equation} \label{equal:dofs}
\dof_i (\uI - \vI) = 0\quad \quad \forall i=1,\dots,\dim(\VnE),\quad \forall \E \in \taun
\end{equation}
where we recall that $\{\dof_i\}_{i=1}^{\dim(\VnE)}$ is the set of degrees of freedom of $\VnE$.
It can be proven that~\eqref{equal:dofs} implies
\begin{equation} \label{15}
\Pinabla \uI = \Pinabla \vI \quad \text{in~$\E \in \taun$.}
\end{equation}
Setting $\qpmo: =\Delta (\uI-\vI) \in \mathbb P_{\p-1}(\E)$, an integration by parts, together with the definitions of~$\uI$ and of~$\vI$ and~\eqref{15}, yields
\begin{equation} \label{estimate:uI-vI}
\begin{split}
\vert \uI - \vI \vert^2_{1,\E}		& \overset{\eqref{equal:dofs}}{=} \int_\E - \qpmo (\uI-\vI) \overset{\eqref{equal:dofs}}{=} \int_\E (I-\Pizpmt)\qpmo \,(\vI-\uI) \\
							& \overset{\eqref{local:enhanced:space}}{=} \int_\E  (I-\Pizpmt) \qpmo \,(\vI-\Pinabla \uI) \overset{\eqref{15}}{=} \int_\E  (I-\Pizpmt) \qpmo \,(\vI-\Pinabla\vI)\\
							& = \int_\E  \qpmo \,(I-\Pizpmt) (\vI - \Pinabla\vI) \le \Vert \qpmo \Vert_{0,\E} \Vert (I-\Pizpmt)(\vn - \Pinabla \vn ) \Vert_{0,\E}\\
							& \overset{\eqref{cazzo}}{\lesssim} \Vert \qpmo \Vert_{0,\E}  \hE (\p-2)^{-1} \Vert \vn - \Pinabla \vn \Vert_{1,\E} \lesssim \hE \p^{-1} \Vert \qpmo \Vert_{0,\E}   \Vert \vn - \Pinabla \vn \Vert_{1,\E},
\end{split}
\end{equation}
where the last but one inequality follows from the properties of the~$L^2$ projector, the fact that we are assuming~$\p > 2$, and Theorem~\ref{theorem:best:polynomial:error}.

Next, we recall the $\p$-polynomial inverse estimate \cite[equation (33)]{hpVEMcorner}
\begin{equation} \label{p:inverse:estimate:on:polygon}
\Vert \qpmo \Vert_{0,\E} = \Vert \Delta (\uI - \vI) \Vert_{0,\E} \lesssim \frac{\p^2}{\hE} \vert \uI - \vI \vert_{1,\E}.
\end{equation}
Combining~\eqref{estimate:uI-vI} and~\eqref{p:inverse:estimate:on:polygon}, we deduce that
\[
\vert \uI - \vI \vert^2_{1,\E} \lesssim \p \vert \uI - \vI \vert_{1,\E} \Vert \vI-\Pinabla\vI \Vert_{1,\E}.
\]
This, together with a Poincar\'e inequality (which applies since $\vI - \Pinabla \vI$ has zero average on~$\partial \E$) and the properties of the projector~$\Pinabla$ (which is the best approximation in~$H^1$), entails
\begin{equation} \label{Masco:estimate}
\begin{split}
\vert \uI - \vI \vert_{1,\E} 	& \lesssim \p \Vert \vI-\Pinabla\vI \Vert_{1,\E}  \lesssim \p \vert\vI - \Pinabla \vI \vert_{1,\E} \le \p \vert \vI - \upi\vert_{1,\E}\\
					& \le \p \left( \vert u -\vI \vert_{1,\E}    + \vert u - \upi\vert_{1,\E} \right).
\end{split}
\end{equation}
Hence, a triangle inequality, together with~\eqref{Chileans:estimate} and~\eqref{Masco:estimate}, leads to
\[
\vert u - \uI \vert_{1,\E} \le \vert u - \vI \vert_{1,\E} + \vert \uI - \vI \vert_{1,\E} \lesssim \p (\vert u - \vI \vert_{1,\E} + \vert u - \upi \vert_{1,\E}) \lesssim \p (\vert u - \phin \vert_{1,\E} + \vert u - \upi \vert_{1,\E}),
\]
which is the claim.
\end{proof}
We have now all the tools so as to prove an $\h\p$-best interpolation error result by means of functions in virtual element spaces.
\begin{cor}[$\h\p$-best interpolation error in virtual element spaces] \label{corollary:best:interpolation:error}
Given~$u \in H^1_0(\Omega)$ with~$u_{|\E}\in H^{s+1}(\E)$ for all~$\E \in \taun$ and for some~$s \ge 1$, there exists~$\uI \in \Vn$ such that
\[
\vert u - \uI \vert_{1,\Omega} \lesssim \frac{\h^{\min(\p,s)}}{\p^{s-1}} \left(  \sum_{\E \in \taun} \Vert u \Vert^2_{s+1,\E}   \right)^{\frac{1}{2}}.
\]
\end{cor}
\begin{proof}
The assertion follows from Theorem~\ref{theorem:best:interpolation:error:auxiliary}, applying \cite[Theorem 4.6]{babuskasurihpversionFEMwithquasiuniformmesh} and Theorem~\ref{theorem:best:polynomial:error} 
to the first and second terms term on the right-hand side of~\eqref{best:interpolation:error}, respectively.
\end{proof}

We point out that the best interpolation error proven in Corollary~\ref{corollary:best:interpolation:error} is suboptimal of one power of~$\p$ with respect to its counterpart in standard VE spaces, see~\cite[Lemma 4.3]{hpVEMbasic}.
As a consequence, it will turn out that performing a pure $\p$-version of the method on a test case with a finite Sobolev regularity solution could lead to a suboptimal rate of convergence.

Notwithstanding, assuming that the target function~$u$ is analytic, the rate of convergence of the $\p$-version of a Galerkin method (such as FEM~\cite{SchwabpandhpFEM} and VEM~\cite{hpVEMbasic}) is typically exponential in terms of the polynomial degree~$\p$;
therefore, the suboptimal polluting factor~$\p$ can be absorbed in the exponential term; see the forthcoming Theorem~\ref{theorem:p:exponential:convergence} for a more precise statement.
In case instead one considers a test case with exact solution having finite Sobolev regularity, one may proceed with $\h\p$-refinement techniques, which lead in any case to exponential convergence, this time in terms of the cubic root of the number of degrees of freedom.
This procedure will be numerically investigated in Section~\ref{subsection:hp-version}.

\subsection{Discrete bilinear forms} \label{subsection:discret:bilinear:forms}
Having recalled that the functions in virtual element spaces are unknown in closed form and therefore, \emph{rebus sic stantibus}, it is not possible to implement the method,
the aim of the present section is to define discrete bilinear forms that are computable via the degrees of freedom of the space.

We begin with the discrete counterpart of the bilinear form~$\a(\cdot, \cdot)$, which is constructed in the spirit of~\cite{bbmr_VEM_generalsecondorderelliptic}.
We first define the local discrete bilinear forms. For all~$\E \in \taun$,
\begin{equation} \label{an}
\anE(\un,\vn) = \sum_{\E \in \taun}  (\K \Pizpmo \nabla \un, \Pizpmo \nabla \vn) + \SEo( (I-\Pinabla)\un, (I-\Pinabla)\vn)  \quad \forall \un,\,\vn \in H^1(\E),\\
\end{equation}
where $\SEo: \ker(\Pinabla)^2\rightarrow \mathbb R$ is any bilinear form computable via the set of local degrees of freedom, satisfying
\begin{equation} \label{stabilization:SE1}
\begin{split}
&\alpha_*(\p) \vert\vn\vert_{1,\E}^2 \le \SEo(\vn,\vn)\quad \forall  \vn \in \VnE \text{ such that } \Pinabla \vn=0,\\
& \SEo(\vn, \vn) \le \alpha^*(\p) \vert \vn \vert_{1,\E}^2 \quad \forall \vn \in H^1(\E) \text{ such that } \Pinabla \vn=0,\\
\end{split}
\end{equation}
for some positive constants $\alpha_*(\p)$ and $\alpha^*(\p)$ independent of $\hE$ but not of $\p$. We recall that~$\Pizpmo \nabla \un$ is explicitly known.
The global discrete bilinear form is instead given by
\[
\an (\un, \vn) 	 = \sum_{\E\in \taun} \anE(\un,\vn) \quad \forall \un,\, \vn \in H^1(\Omega).
\]
The following result concerns the continuity and the coercivity of~$\anE$.
\begin{lem} \label{lemma:continuity:and:coercivity:an}
For all~$\E \in \taun$, the local discrete bilinear form~$\anE$ in~\eqref{an} satisfies the following bounds:
\begin{equation} \label{properties:an}
\begin{split}
& \min(\kcal _*, \alpha_*(\p)) \vert \vn \vert_{1,\E}^2 \le \anE (\vn, \vn) \quad \forall \vn \in \VnE,\\
& \anE (\vn, \vn) \le (\kcal^* + \alpha^*(\p)) \vert \vn \vert_{1,\E}^2 \quad \forall \vn \in H^1(\E),\\
\end{split}
\end{equation}
where we recall that~$\kcal_*$ and~$\kcal^*$ are introduced in~\eqref{property_K}, whereas~$\alpha_*(\p)$ and~$\alpha^*(\p)$ are defined in~\eqref{stabilization:SE1}.
\end{lem}
\begin{proof}
We begin with the upper bound:
\[
\begin{split}
\anE (\vn, \vn) 	& = (\K \Pizpmo \nabla \vn, \Pizpmo \nabla\vn)_{0,\E} + \SEo( (I-\Pinabla) \vn, (I-\Pinabla)\vn) \\
			& \le \kcal^*  \Vert \Pizpmo \nabla \vn \Vert_{0,\E}^2 + \alpha^*(\p)  \vert (I-\Pinabla) \vn \vert_{1,\E}^2 \le \kcal^*  \Vert \nabla \vn \Vert_{0,\E}^2 + \alpha^*(\p)  \Vert \nabla  \vn \Vert_{0,\E}^2 \\
			& \le  (\kcal^* + \alpha^*(\p))  \vert \vn \vert_{1,\E}^2.
\end{split}
\]
For what concerns the lower bound, we have
\[
\begin{split}
\anE (\vn, \vn) 	& = (\K \Pizpmo \nabla \vn, \Pizpmo \nabla\vn)_{0,\E} + \SEo( (I-\Pinabla)\vn, (I-\Pinabla)\vn) \\
			& \ge \kcal _* \Vert \Pizpmo \nabla \vn \Vert_{0,\E}^2 + \alpha_*(\p) \vert (I-\Pinabla) \vn \vert_{1,\E}^2\\
			& \ge \kcal _* \Vert \Pizpmo \nabla \vn \Vert_{0,\E}^2 + \alpha_*(\p) \Vert (I-\Pizpmo) \nabla \vn \Vert_{0,\E}^2\\
			& \ge \min(\kcal _*, \alpha_*(\p)) \vert \vn \vert_{1,\E}^2.
\end{split}
\]
\end{proof}
For what concerns the discrete counterpart of~$\b(\cdot, \cdot)$, we pick
\begin{equation} \label{bn}
\bn(\un, \vn) = \sum_{\E\in \taun} \bnE(\un,\vn) \quad \forall \un,\,\vn \in H^1(\Omega),
\end{equation}
where, for all~$\E \in \taun$,
\[
\bnE(\un,\vn) =(\V \Piz \un, \Piz \vn) _{0,\E} \quad \forall \un,\,\vn \in H^1(\E),
\]
which are computable, owing to the fact that~$\Pizpmo$ is available in closed form, from the degrees of freedom; we recall that we are assuming to be able to compute exactly integrals of given smooth functions (otherwise, a sufficiently good quadrature formula would suffice).

Finally, we focus on the discrete counterpart of~$\c(\cdot, \cdot)$:
\begin{equation} \label{cn}
\cn(\un, \vn) = \sum_{\E\in \taun} \cnE(\un,\vn) \quad \forall \un,\, \vn \in H^1(\Omega)
\end{equation}
where, for all~$\E \in \taun$,
\[
\cnE(\un,\vn) = (\Piz \un, \Piz \vn)_{0,\E} +\SEz((I-\Piz)\un, (I-\Piz)\vn) \quad \forall \un,\,\vn \in H^1(\E),\\
\]
and where~$\SEz: \ker(\Piz)^2\rightarrow \mathbb R$ is a bilinear form computable via the set of the local degrees of freedom, such that
\begin{equation} \label{stabilization:SE0}
\SEz (\vn,\vn) \ge \beta_*(\p) \Vert \vn \Vert_{0,\E}^2  \quad \forall \vn\in \ker(\Piz),
\end{equation}
and such that
\begin{equation} \label{continuity:S0}
\SEz (\vn-\Piz\vn, \vn-\Piz \vn) \le \hE^2 \beta^* \vert \vn -\Pinablapmo \vn \vert_{1,\E}^2\quad \forall \vn \in H^1(\E),
\end{equation}
for some positive constants~$\beta_*(\p)$ independent of~$\hE$ but not of~$\p$, and~$\beta^*$ independent of~$\hE$ and~$\p$.

The following result concerns the continuity and the coercivity of~$\cn$.
\begin{lem} 
The discrete bilinear form~$\cn$ in~\eqref{cn} satisfies the two following properties:
\begin{equation} \label{properties:cn}
\begin{split}
& \min(1, \beta_*(\p)) \Vert \vn \Vert_{0,\E}^2 \le \cnE (\vn, \vn) \quad \forall \vn \in \VnE,\\
& \cnE (\vn, \vn) \le \max(1 , \beta^*) (\Vert \vn \Vert_{0,\E}^2+ \hE^2 \vert \vn \vert _{1,\E}^2) \quad \forall \vn \in H^1(\E),\\
\end{split}
\end{equation}
where we recall that~$\beta_*(\p)$ and~$\beta^*$ are introduced in~\eqref{stabilization:SE0} and~\eqref{continuity:S0}, respectively.
\end{lem}
\begin{proof}
The proof of the lower bound is the same as that of its counterpart in Lemma~\ref{lemma:continuity:and:coercivity:an}. For what concerns the upper bound, we proceed as follows:
\[
\begin{split}
\cnE (\vn, \vn) 	& = (\Piz \vn, \Piz \vn)_{0,\E} + \SEz( (I-\Piz)\vn, (I-\Piz)\vn)\\
			& \le \Vert \vn \Vert_{0,\E}^2 + \hE^2\beta^* \vert \vn - \Pinablapmo \vn \vert_{1,\E} ^2 \le \Vert \vn \Vert_{0,\E}^2 + \hE^2\beta^*  \vert \vn \vert_{1,\E}^2 \\
			& \le \max(1,  \beta^*) (\Vert \vn \Vert_{0,\E}^2 + \hE^2 \vert \vn \vert_{1,\E}^2 ).
\end{split}
\]
\end{proof}

\begin{remark} 
Following~\cite{VEM_eig_basic}, one could in principle construct a method by removing the stabilization~$\SEz$.
The reason for which we stabilize the bilinear form~$\cn$ is simply that otherwise the resulting matrix could be singular.
Notwithstanding, we experienced numerically that employing the stabilization leads to better performance of the routines for the solution of generalized eigenvalue problems.
\end{remark}

\subsubsection{Explicit choices for the stabilizations $\SEz$ and $\SEo$} \label{subsubsection:stabilizations}
In this section, we introduce two \emph{explicit} stabilizing bilinear forms~$\SEo$ and~$\SEz$, see~\eqref{stabilization:SE1} and~\eqref{stabilization:SE0}, respectively, with explicit continuity and coercivity bounds in terms of~$\p$
on~$\alpha_*(p)$, $\alpha^*(\p)$, and $\beta_*(\p)$.

\begin{thm} \label{theorem:stabilization:SE1}
Given
\begin{equation} \label{explicit:stabilization:SE1}
\SEo(\un,\vn) = \frac{\p^2}{\hE^2} (\Pizpmt \un, \Pizpmt \vn)_{0,\E} + \frac{\p}{\hE} (\un, \vn)_{0,\partial\E},
\end{equation}
the following bounds on the constants $\alpha_*(\p)$ and $\alpha^*(\p)$ in~\eqref{stabilization:SE1} hold true:
\begin{equation} \label{bound:alpha}
\alpha_*(\p) \gtrsim \p^{-5},\quad \alpha^*(\p) \lesssim \p^{2}.
\end{equation}
\end{thm}
\begin{proof}
The proof is analogous to that of~\cite[Theorem 2]{hpVEMcorner}. Notwithstanding, we show a few details, since here we employ enhanced virtual element spaces,
which are slightly different from the standard ones of~\cite{hpVEMcorner}.

More precisely, we only show some details regarding the bound on $\alpha_*(\p)$. Given $\E \in \taun$ and $\vn\in \ker(\Pinabla)$, it holds true that
\[
\begin{split}
\vert \vn \vert^2_{1,\E}	& = \int_\E \nabla \vn \cdot \nabla \vn = \int_{\E} -\Delta \vn\, \vn + \int_{\partial \E} \partial_\n \vn\, \vn  = \int_{\E} -\Delta \vn\, \Piz \vn + \int_{\partial \E} \partial_\n \vn\, \vn.\\
\end{split}
\]
Using the enhancing constraints in~\eqref{local:enhanced:space} and the fact that~$\vn$ belongs to~$\ker(\Pinabla)$, we deduce
\[
\begin{split}
\vert \vn \vert^2_{1,\E}	& = \int_{\E} -\Delta \vn \, \Pizpmt \vn + \int_{\partial \E} \partial_\n \vn \, \vn\\
					& \le \Vert \Delta \vn \Vert_{0, \E} \Vert \Pizpmt\vn\Vert_{0, \E}  + \Vert \partial_\n \vn\Vert_{-\frac{1}{2},\partial \E} \Vert \vn\Vert_{\frac{1}{2},\partial\E}.
\end{split}
\]
Having this, it suffices to proceed as in~\cite[Theorem 2]{hpVEMcorner}.
\end{proof}
We point out that the bound on~$\alpha^*(\p)$ can be actually improved, see~\cite[Theorem 2]{hpVEMcorner};
however, as the topic is bristly with technicalities, we avoid further technicalities and notations, sticking rather to the bounds in~\eqref{bound:alpha}.

\medskip
Next, we introduce a stabilization~$\SEz$ satisfying the properties~\eqref{stabilization:SE0} and~\eqref{continuity:S0}.
\begin{thm} \label{theorem S0}
Given
\begin{equation} \label{explicit:stabilization:S0}
\SEz( \un, \vn) = \frac{\hE}{\p^2} (\un,\vn)_{0, \partial \E},
\end{equation}
the following bounds on~$\beta_*(\p)$ and~$\beta^*$ introduced in~\eqref{stabilization:SE0} and~\eqref{continuity:S0}, respectively, hold true:
\[
\beta_*(\p) \gtrsim \p^{-6},\quad  \beta^* \lesssim 1.
\]
\end{thm}
\begin{proof}
We assume without loss of generality that~$\hE=1$, since the general assertion follows from a scaling argument,
and that~$\p \ge 2$, since we are interested in the asymptotic behavior in terms of~$\p$.

We begin with the bound on~$\beta_*(\p)$. Given $\un \in \ker(\Piz)$, we apply Theorem~\ref{theorem:best:polynomial:error} to show
\begin{equation} \label{formula:*}
\Vert \un \Vert _ {0,\E} \lesssim \p^{-1}\vert \un \vert_{1,\E}.
\end{equation}
Integrating by parts, using the fact that $\un \in \ker(\Pi^0_{\p-1})$, applying the definition of the $H^{-\frac{1}{2}}(\partial \E)$ norm,
applying the Neumann trace inequality \cite[Theorem A.33]{SchwabpandhpFEM}, and applying the~$\p$-inverse inequality already employed in~\eqref{p:inverse:estimate:on:polygon}, we get
\[
\begin{split}
\vert \un \vert^2_{1,\E} 	& = \int_\E -(\Delta\un) \un + \int_{\partial \E} \partial_\n \un\, \un = \int_{\partial \E} \partial_\n \un \, \un \le \Vert \partial_\n \un \Vert_{-\frac{1}{2}, \partial \E} \Vert \un \Vert_{\frac{1}{2}, \partial \E}  \\
					& \lesssim \left(\vert \un \vert_{1,\E} + \Vert \Delta \un \Vert_{0,\E} \right) \Vert \un  \Vert_{\frac{1}{2}, \partial \E}  \lesssim \p^2 \vert \un \vert_{1,\E} \Vert \un \Vert_{\frac{1}{2}, \partial \E},
\end{split}
\]
whence
\begin{equation} \label{formula:**}
\vert \un \vert_{1,\E} \lesssim \p^2 \Vert \un \Vert_{\frac{1}{2}, \partial \E}.
\end{equation}
Using that~$\un$ is piecewise polynomial over~$\partial \E$ and applying the one dimensional $\p$-inverse inequality \cite[Theorem 3.91]{SchwabpandhpFEM} together with interpolation theory~\cite[Appendix B]{SchwabpandhpFEM}, we deduce
from~\eqref{formula:*} and~\eqref{formula:**} that
\[
\Vert \un \Vert_{0,\E} \lesssim \p^2 \Vert \un \Vert_{0,\partial \E},
\]
which is the claim.

\medskip
For what concerns instead the bound on~$\beta^*$, one has, for all~$\un \in H^1(\E)$,
\[
\begin{split}
\Vert \un -\Piz \un \Vert_{0,\partial \E} \le \Vert \un - \Pinablapmo \un \Vert_{0,\partial \E} + \Vert  \Pinablapmo \un - \Piz\un  \Vert_{0,\partial \E}.
\end{split}
\]
Concerning the first term on the right-hand side, we apply a trace and a Poincar\'e inequality; concerning the second one, we apply the $\p$-trace inverse inequality \cite[equation (4.6.4)]{SchwabpandhpFEM} on every triangle in 
the subtriangulation $\tautilden(\E)$ obtained by connecting the center of any of the maximal balls with respect to which~$\E$ is star-shaped to the vertices of~$\E$, getting
\[
\Vert \un -\Piz \un \Vert_{0,\partial \E} \lesssim \vert \un - \Pinabla\un \vert_{1,\E} + \p \Vert  \Pinablapmo \un - \Piz\un  \Vert_{0,\E} \lesssim \p \vert  \un - \Pinablapmo \un \vert_{1, \E},
\]
which entails the assertion.
\end{proof}

The dependence in terms of~$\p$ of the stability constants of~$\SEo$ and~$\SEz$ seems to be very large and will play a role also in the convergence estimates of the method, see Theorem~\ref{theorem:Tn:tends:to:T}.
In particular, the convergence rate could be polluted by some powers of~$\p$, see for instance the estimates in Theorem~\ref{theorem:convergence:source:problem}. However:
\begin{itemize}
\item the dependence that we have theoretically pinpointed is in principle pessimistic. In the setting of standard (i.e., nonenhanced) VEM, such dependence was proven to be much milder in practice, see \cite[Section 6.4]{hpVEMbasic} and \cite[Section 4.1]{hpVEMcorner};
\item also in the worst possible scenario, i.e., even assuming that the bounds on $\alpha_*(\p)$ and $\alpha^*(\p)$ were sharp,
it is possible to show that the $\p$- (for analytic eigenfunctions) and the $\h\p$-versions (for eigenfunctions with finite Sobolev regularity) of the method lead in any case to exponential convergence
in terms of~$\p$ and in the cubic root of the number of degrees of freedom,
see Theorem~\ref{theorem:convergence:eigenfunctions} and Section~\ref{subsection:hp-version}, respectively.
\end{itemize}

\subsection{The virtual element method} \label{subsection:the:method}
Having described the approximation spaces and the discrete bilinear forms, we define the method associated with the eigenvalue problem~\eqref{continuous:problem:weak}:
\begin{equation} \label{VEM:eigenvalue}
\begin{cases}
\text{find }(\lambdan,\un) \in \mathbb R \times \Vn \text{ such that } \Vert \un \Vert_{0,\Omega} =1\\
\an(\un,\vn) + \bn(\un,\vn)= \lambdan \cn(\un,\vn) \quad \forall \vn \in \Vn.\\
\end{cases}
\end{equation}
The method associated with the source problem~\eqref{weak:formulation:continuous:source} is instead the following: given $\f \in H^1(\Omega)$,
\begin{equation} \label{VEM:source}
\begin{cases}
\text{find } \un \in \Vn \text{ such that}\\
\an(\un,\vn) + \bn(\un,\vn)=  \cn(\f,\vn) \quad \forall \vn \in \Vn.\\
\end{cases}
\end{equation}
The method~\eqref{VEM:eigenvalue} is well-posed thanks to the coercivity of the bilinear form on the left-hand side, which follows from~\eqref{properties:an} and~\eqref{bn},
and the continuity of the right-hand side with respect to the~$\Vert \cdot\Vert_{1,\E}$-norm, see~\eqref{properties:cn}.

We also define the solution operator $\Tn \in \mathcal L(H^1(\Omega))$ as
\begin{equation*}
\Bn(\Tn \f, \vn)  = \cn(\f, \vn)\quad \forall \vn \in \Vn,
\end{equation*}
where we have set $\Bn(\cdot, \cdot) = \an(\cdot, \cdot) + \bn(\cdot, \cdot)$.

Analogously to the continuous case, the operator~$\Tn$ is self-adjoint, compact (since the image of~$\Tn$ has finite dimension), and positive definite.
Besides, given $(\lambdan, \wn)$ an eigenpair of~\eqref{VEM:eigenvalue}, one can prove that $(\frac{1}{\lambdan}, \wn)$ is an eigenpair of the discrete solution operator~$\Tn$.

\section{Convergence analysis of the $\p$-version} \label{section:convergence:analysis}
This section is devoted to show the convergence of the discrete eigenvalues and eigenfunctions to the continuous ones, when employing the $\p$-version of the method.
A particular emphasis is stressed on the case of analytic eigenfunctions, where exponential convergence in terms of~$\p$ is proven.
The exponential convergence in terms of the cubic root of the number of degrees of freedom for singular functions is not theoretically covered in the present paper, but will be the objective of a numerical investigation in Section~\ref{subsection:hp-version}.

The remainder of the section is organized as follows. In Section~\ref{subsection:some:auxiliary:results}, we introduce some technical results and prove $\p$-exponential convergence on the source problem~\eqref{weak:formulation:continuous:source} for analytic solutions;
we investigate instead the approximation of the eigenpairs (with tools stemming from the Babu\v ska-Osborn theory) in Section~\ref{subsection p spectral approximation for compact operators}.

\subsection{Some auxiliary results} \label{subsection:some:auxiliary:results}
We first prove an approximation result on the continuous~\eqref{weak:formulation:continuous:source} and the discrete~\eqref{VEM:source} source problems.
\begin{thm} \label{theorem:Tn:tends:to:T}
Given $\f \in H^1(\Omega)$, let~$u$ and~$\un$ be the solutions to the continuous and discrete source problems~\eqref{weak:formulation:continuous:source} and~\eqref{VEM:source}, respectively. Then, the following bound holds true:
\[
\begin{split}
\vert u -\un \vert_{1,\Omega} 			& \lesssim \mu_1(\p) \hE^2 \vert \f - \Pinablapmo \f \vert_{1,\taun} + \mu_2(\p) \vert u -\uI\vert_{1,\Omega}+\mu_3(\p) \vert u - \upi \vert_{1,\taun} \\
								& \quad + \mu_4 \Vert u - \Piz u \Vert_{0,\Omega} + \mu_5(\p) \Vert \K \nabla u - \Pizpmo (\K \nabla u)\Vert_{0,\Omega} + \mu_5(\p) \Vert \V u - \Piz (\V u) \Vert_{0,\Omega},
\end{split}
\]
where
\[
\begin{split}
& \mu_1 (\p) = \max (\kcal^{-1}_*, \alpha^{-1}_*(\p))  \max(1, \beta^*),\\
& \mu_2 (\p) = 1 + \max (\kcal^{-1}_*, \alpha^{-1}_*(\p)) \left[ \kcal^* + \alpha^*(\p)  +  \nu^*  \right]  , \\
& \mu_3 (\p) = \max (\kcal^{-1}_*, \alpha^{-1}_*(\p)) \left[ \alpha^* (\p)   \right], \\
& \mu_4 = \max (\kcal^{-1}_*, \alpha^{-1}_*(\p)) \nu^* \\ 
& \mu_5 (\p) = \max (\kcal^{-1}_*, \alpha^{-1}_*(\p)), \\
\end{split}
\]
being $\alpha_*(\p)$ and $\alpha^*(\p)$ defined in~\eqref{stabilization:SE1}, $\beta^*$ being defined in~\eqref{continuity:S0}, $\kcal_*$ and $\kcal^*$ being defined in~\eqref{property_K}, and $\nu^*$ being defined in~\eqref{property_V},
and having set $\Piz (u)_{|\E} = \Pi^{0,\E}_{\p-1} (u_{|\E})$ and $\Pinabla(u)_{|\E} = \Pi^{\nabla,\E}_\p (u_{|\E})$.

\end{thm}
\begin{proof}
Setting $\deltan = \un-\uI$, we apply~\eqref{properties:an}, we use the positiveness of~$\bn$ in~\eqref{bn}, and we perform some computations, getting
\[
\begin{split}
\min (\kcal_* ,\alpha_*(\p)) \vert \deltan \vert_{1,\Omega}^2	& \le  \an(\deltan, \deltan)  \le \an(\deltan, \deltan) + \bn(\deltan, \deltan) \\
												& = \an(\un, \deltan) + \bn (\un, \deltan) - \sum_{\E\in \taun} \{ \anE(\uI, \deltan) + \bn(\uI, \deltan) \}.
\end{split}
\]
We note that, for all $\E\in \taun$,
\[
\begin{split}
\anE(\uI, \deltan) + \bn(\uI, \deltan) 	& = \anE(\uI-u,\deltan) + \anE(u, \deltan) - \aE(u,\deltan)\\
							& \quad +\aE(u,\deltan) + \bE(u,\deltan) + \bnE(\uI-u, \deltan) + \bnE(u,\deltan) - \bE(u,\deltan).
\end{split}
\]
Thus, recalling~\eqref{weak:formulation:continuous:source} and~\eqref{VEM:source}, we obtain
\begin{equation} \label{main:equation:abstract:error:analysis}
\begin{split}
\min(\kcal_*, \alpha_*(\p)) \vert \deltan \vert_{1,\Omega}^2\le  	&\underbrace{\cn(\f,\deltan) - \c(\f ,\deltan)}_{A}     -\sum_{\E\in\taun} \{  \underbrace{\anE(\uI-u,\deltan)}_{B^\E} + \underbrace{\anE(u, \deltan) - \aE(u,\deltan)}_{C^\E}  \\
												& \quad\quad\quad\quad  \quad\quad\quad\quad  \quad\quad\quad + \underbrace{\bnE(\uI-u, \deltan)}_{D^\E} + \underbrace{\bnE(u,\deltan) - \bE(u,\deltan)}_{E^\E} \}.
\end{split}
\end{equation}
We bound the five terms on the right-hand side of~\eqref{main:equation:abstract:error:analysis} separately.

We begin with the first one.
Applying the definitions of~$\c$ and~$\cn$ in~\eqref{notation:for:bilinear:forms} and~\eqref{cn}, respectively,
using~\eqref{continuity:S0} and the properties of the~$L^2$ projector, and applying a Poincar\'e inequality, we deduce
\[
\begin{split}
A 	& =\cn(\f,\deltan) - \c(\f ,\deltan) \\
	& = \sum_{\E \in \taun} \left( (\Piz\f, \Piz \deltan)_{0,\E} + \SEz( (I-\Piz) \f, (I-\Piz) \deltan) - (\f,\deltan)_{0,\E} \right) \\
	& \le \sum_{\E \in \taun} \left( \hE^2 \beta^* \vert \f - \Pinablapmo \f \vert_{1,\E} \vert \deltan - \Pinabla \deltan \vert_{1,\E}  + \Vert \f -\Piz \f \Vert_{0,\E} \Vert \deltan - \Piz \deltan \Vert_{0,\E}  \right)  \\
	& \le \sum_{\E \in \taun} \max(1,\beta^*) \hE^2 \vert \f - \Pinablapmo \f\vert_{1,\E} \vert \deltan - \Pinablapmo \deltan \vert_{1,\E} \\
	& \le \sum_{\E \in \taun} \max(1,\beta^*) \hE^2 \vert \f - \Pinablapmo \f\vert_{1,\E} \vert \deltan \vert_{1,\E}.\\
\end{split}
\]
For what concerns the second local term, Lemma~\ref{lemma:continuity:and:coercivity:an} again yields
\[
B^\E=\anE(\uI-u,\deltan) \le (\kcal^* +  \alpha^*(\p)) \vert u - \uI \vert_{1,\E} \vert \deltan \vert_{1,\E}.
\]
Regarding the third local term, we apply the definition of~\eqref{notation:for:bilinear:forms} and~\eqref{an}, respectively, the properties of the~$L^2$ projector, \eqref{property_K}, and~\eqref{stabilization:SE1}, getting
\[
\begin{split}
& C^\E=\anE(u, \deltan) - \aE(u,\deltan)	\\
& = (\K \Pizpmo \nabla u, \Pizpmo  \nabla \deltan)_{0,\E} + \SEo ((I-\Pinabla) u, (I-\Pinabla) \deltan) - (\K \nabla u, \nabla \deltan)_{0,\E}\\
& = (\K (\Pizpmo \nabla u - \nabla u), \Pizpmo  \nabla \deltan)_{0,\E} - (\K \nabla u - \Pizpmo (\K \nabla u), \nabla \deltan - \Pizpmo  \nabla \deltan)_{0,\E} \\
& \quad +\SEo ((I-\Pinabla) u, (I-\Pinabla) \deltan)\\
& \le \left(\kcal^* \Vert \nabla u - \Pizpmo \nabla u \Vert_{0,\E}  + \Vert \K \nabla u - \Pizpmo (\K \nabla u) \Vert_{0,\E} + \alpha^*(\p) \vert u - \Pinabla u \vert _{1,\E}\right) \vert \deltan \vert_{1,\E}.
\end{split}
\]
The fourth local term can be bounded using~\eqref{property_V} and~\eqref{bn}:
\[
\begin{split}
D^\E=\bnE(\uI-u, \deltan) \le \nu^* \Vert u - \uI \Vert _{0,\E} \Vert \deltan \Vert_{0,\E}.
\end{split}
\]
Eventually, we deal with the fifth local term, which can be bounded employing the definitions of~$\bnE$ and~$\bn$ in~\eqref{notation:for:bilinear:forms} and~\eqref{bn}, respectively, and~\eqref{property_V}:
\[
\begin{split}
E^\E=\bnE(u,\deltan) - \bE(u,\deltan) 	& = ( \V \Piz u, \Piz \deltan)_{0,\E} - (\V u, \deltan )_{0,\E} \\
						& = ( \V (\Piz u - u), \Piz \deltan)_{0,\E} - (\V u, \deltan -\Piz \deltan)_{0,\E}\\
						& \le (\nu^*\Vert u - \Piz u \Vert _{0,\E}  + \Vert \V u - \Piz \V u \Vert_{0,\E})\Vert \deltan\Vert_{0,\E}.\\
\end{split}
\]
Collecting the five bounds above in~\eqref{main:equation:abstract:error:analysis}, applying an $\ell^2$ Cauchy-Schwarz inequality, and applying a Poincar\'e inequality on~$\Omega$, yield
\[
\begin{split}
\min(\kcal_*, \alpha_*(\p)) & \vert \deltan \vert_{1,\Omega} \lesssim  \max(1,\beta^*) \hE^2 \vert \f - \Pinablapmo \f \vert_{1,\taun} + (\kcal^* + \alpha^*(\p)) (\vert u - \uI \vert_{1,\Omega})\\
				& + \kcal^* \Vert \nabla u - \Pizpmo \nabla u \Vert_{0,\taun} + \Vert \K\nabla u - \Pizpmo (\K\nabla u)\Vert_{0,\Omega} + \alpha^*(\p) \vert u - \upi \vert_{1,\taun} \\
				& + \nu^* \vert u - \uI \vert_{1,\Omega} + \nu^* \Vert u - \Piz u \Vert_{0,\Omega} +  \Vert \V u - \Piz (\V u) \Vert_{0,\Omega}.
\end{split}
\]
The assertion follows by noting that
\[
\vert u - \un \vert_{1,\Omega} \le \vert u - \uI \vert _{1,\Omega} + \vert \deltan \vert_{1,\Omega}.
\]
\end{proof}
Best polynomial approximation and best interpolation results entail the following theorem, which deals with the convergence rate of the error in the approximation of the source problem~\eqref{weak:formulation:continuous:source}.
\begin{thm} \label{theorem:convergence:source:problem}
Let~$u$ and~$\f$ be the solution and the right-hand side of problem~\eqref{weak:formulation:continuous:source}, and assume that they belong to~$H^1_0(\Omega)$ and~$H^1(\Omega)$, respectively,
and let their restriction on every element $\E\in \taun$ belong to $H^{s+1}(\E)$, $s\ge 0$.
Then, recalling that the coefficients~$\K$ and~$\V$ in~\eqref{weak:formulation:continuous:source} are piecewise analytic over~$\taun$, see assumption (\textbf{A}), it holds that
\[
\begin{split}
\vert	& u - \un \vert_{1,\Omega} \\
	& \lesssim \frac{\max(1,\kcal^* + \alpha^*(\p), \beta^*, \nu^*)}{\min(\kcal_*,\alpha_*(\p))} \frac{\h^{\min(\p,s)}}{\p^{s-1}} \left( \h^2\Vert \f \Vert_{s+1,\taun} + \Vert u \Vert_{s+1,\taun} + \Vert \V u \Vert_{s+1,\taun} + \Vert \K \nabla u \Vert_{s,\taun}  \right).
\end{split}
\]
\end{thm}
\begin{proof}
It suffices to combine Theorem~\ref{theorem:best:polynomial:error}, Corollary~\ref{corollary:best:interpolation:error}, and Theorem~\ref{theorem:Tn:tends:to:T}.
\end{proof}
It is clear from the estimate in Theorem~\ref{theorem:convergence:source:problem} that, whereas the $\h$-version of the method converges optimally,
the $\p$-version, whenever the solution to problem~\eqref{weak:formulation:continuous:source} has finite Sobolev regularity, it does not.
On the one hand, we showed in Corollary~\ref{corollary:best:interpolation:error} that in the enhanced VEM framework, the best interpolation estimates are suboptimal of one power;
on the other, one also has to pay additional powers of $\p$ due to the effects of the local stabilizations~$\SEo$ and~$\SEz$.

At any rate, if~$u$ is the restriction on~$\Omega$ of an analytic solution, the rate of convergence in terms of~$\p$, is exponential, as stated in the following result.
\begin{thm} \label{theorem:p:exponential:convergence}
Let~$u$ and~$\un$ be the solutions to~\eqref{weak:formulation:continuous:source} and~\eqref{VEM:source}, respectively; moreover, assume that~$u$ is the restriction on~$\Omega$ of an analytic function defined on a sufficiently large extension of~$\Omega$.
Then,
\[
\vert u - \un \vert_{1,\Omega} \lesssim  \exp(-b\, \p),
\]
for some positive constant~$b$ independent of the discretization parameters.
\end{thm}
\begin{proof}
Having at disposal Theorem~\ref{theorem:convergence:source:problem}, it suffices to apply the argument of~\cite[Section~5]{hpVEMbasic}.
\end{proof}
One of the main point behind the proof of Theorem~\ref{theorem:p:exponential:convergence} is that the algebraic losses in terms of~$\p$, due to the suboptimality of the best interpolation estimates and to the presence of the stabilizing parameters $\alpha_*(\p)$ and $\alpha^*(\p)$,
are absorbed in a term which is exponentially decreasing in~$\p$.
\medskip

It still remains open the issue of how to proceed in case the solution is not analytic, if one wants to recover some sort of exponential convergence of the error.
In fact, the instance of finite Sobolev regularity solutions will be numerically addressed in Section~\ref{section:numerical:results},
where the $\h\p$-version of the method will be considered, and exponential convergence in terms of the cubic root of the number of degrees of freedom will be shown.
A theoretical analysis of this approach is actually doable, e.g. following the lines of~\cite[Section 5]{hpVEMcorner}, but is not addressed in the present paper.

In the remainder of Section~\ref{section:convergence:analysis}, we assume therefore that~$u$ fulfills the assumptions of Theorem~\ref{theorem:p:exponential:convergence}.
Moreover, we will focus on the $\p$-version only, since the $\h$-analysis was the topic of~\cite{VEM_eig_Schroedinger}.

\subsection{$\p$-spectral approximation for compact operators} \label{subsection p spectral approximation for compact operators}
We proceed here with the convergence analysis of the eigenfunctions and the eigenvalues of the solution operator~$\T$.
We will employ the tools of the Babu\v ska-Osborn theory for compact operators~\cite{BabuskaOsborn}.

More precisely, given~$\T$ the solution operator associated with the problem~\eqref{weak:formulation:continuous:source},
and~$\Tn$, the sequence of compact solution operators associated with the method~\eqref{VEM:source}, the condition
\begin{equation} \label{condition:T:Tn}
\Vert \T - \Tn \Vert_{\mathcal L(H^1(\Omega))} \longrightarrow 0 \quad \text{as}\quad n\longrightarrow +\infty,
\end{equation}
is sufficient, see~\cite[Proposition 7.4]{boffiAN}, see also~\cite{BabuskaOsborn}, in order to get the two following facts:
assuming that~$\h$ and~$\p$ are sufficiently small and large, respectively,
\begin{itemize}
\item given~$\lambda$ an exact eigenvalue in~\eqref{continuous:problem:weak} with multiplicity~$m$, method~\eqref{VEM:eigenvalue} provides precisely~$m$ discrete eigenvalues converging to~$\lambda$;
\item given~$\lambda_n$ a discrete eigenvalue, $\lambda_n$ converges to one continuous eigenvalue.
\end{itemize}
Importantly, condition~\eqref{condition:T:Tn} is also necessary in order to prove spectral approximation properties, see~\cite{boffi2000problem}.

Thus, we begin with the following result, which provides~\eqref{condition:T:Tn} for the $\p$-version of VEM, in case the solution to the source problem \eqref{weak:formulation:continuous:source} is the restriction of an analytic function over a sufficiently large extension of the domain~$\Omega$.
\begin{lem} \label{lemma:approximation:solution:operator}
Under the assumptions of Theorem~\ref{theorem:p:exponential:convergence}, it holds that
\[
\Vert \T - \Tn \Vert_{\mathcal L(H^1(\Omega))} \lesssim \exp(-b \, \p),
\]
for some positive constant~$b$ independent of the discretization parameters.
\end{lem}
\begin{proof}
It suffices to observe that
\[
\begin{split}
\Vert \T - \Tn \Vert_{\mathcal L(H^1(\Omega))} 	& = \sup_{\f \in H^1(\Omega), \Vert \f \Vert_{1,\Omega}=1} \Vert \T \f - \Tn \f \Vert_{1,\Omega} = \sup_{\f \in H^1(\Omega), \Vert \f \Vert_{1,\Omega}=1} \Vert u -\un \Vert_{1,\Omega}\\
& \vspace{0.5mm}\\
									& \le (1+c_P(\Omega))  \sup_{\f \in H^1(\Omega), \Vert \f \Vert_{1,\Omega}=1} \vert u - \un \vert_{1,\Omega},\\
\end{split}
\]
where $c_P(\Omega)$ is the Poincar\'e constant on $\Omega$, and then apply Theorem~\ref{theorem:p:exponential:convergence}.
\end{proof}
The convergence rate for eigenfunctions is a consequence of Lemma~\ref{lemma:approximation:solution:operator} and the Babu\v ska-Osborn theory.
In particular, given~$\lambda$ a continuous eigenvalue with multiplicity~$m$, and given $\lambda_{1,n}$, \dots, $\lambda_{m,n}$ the associated discrete eigenvalues,
we first introduce the gap between $\mathcal E_{\lambda, n}$ (the direct sum of the eigenspaces generated by $\lambda_{1,n}$, \dots, $\lambda_{m,n}$) and~$\mathcal E_\lambda$ (the eigenspace generated by $\lambda$):
\begin{equation} \label{gap}
\widehat \delta (\mathcal E_\lambda,  \mathcal E_{\lambda,n}) = \max \left( \delta (\mathcal E_\lambda,  \mathcal E_{\lambda,n}), \delta (\mathcal E_{\lambda,n}, \mathcal E_\lambda)   \right),
\end{equation}
where we are using the notation
\[
\delta (\mathbf X, \mathbf Y) = \sup _{x\in \mathbf X, \, \Vert x \Vert_{1,\Omega}=1} \left( \inf_{y\in \mathbf Y} \Vert x-y \Vert_{1,\Omega}   \right).
\]
The following bound on the gap~$\widehat \delta (\mathcal E_\lambda,  \mathcal E_{\lambda,n}) $ is valid.
\begin{thm} \label{theorem:convergence:eigenfunctions}
Let~$u$, the continuous eigenfunction corresponding to the eigenvalue~$\lambda$, be the restriction on~$\Omega$ of
an analytic function defined on a sufficiently large extension of~$\Omega$, then it holds true that
\[
\widehat \delta (\mathcal E_\lambda,  \mathcal E_{\lambda,n}) \lesssim \exp(-b \p),
\]
for some positive constant~$b$ independent of the discretization parameters.
\end{thm}
\begin{proof}
The assertion follows from the Babu\v ska-Osborn theory~\cite[Theorem 7.1 and 7.3]{BabuskaOsborn}, which states that
\[
\widehat \delta (\mathcal E_\lambda,  \mathcal E_{\lambda,n}) \lesssim  \Vert (\T - \Tn)_{|\mathcal E_\lambda} \Vert_{\mathcal L(H^1(\Omega))},
\]
and applying Lemma~\ref{lemma:approximation:solution:operator}.
\end{proof}
\begin{remark} \label{remark:convergence:eigenfunctions}
Owing to the definition~\eqref{gap}, Theorem~\ref{theorem:convergence:eigenfunctions} entails the $\p$-exponential convergence of the eigenfunctions.
\end{remark}
Finally, we address the convergence of the eigenvalues.
\begin{thm} \label{theorem:convergence:eigenvalues}
Under the assumptions of Theorem~\ref{theorem:p:exponential:convergence},
let $\lambda_n$ be an eigenvalue of problem~\eqref{VEM:eigenvalue}, converging to the eigenvalue~$\lambda$ of problem~\eqref{continuous:problem:weak}.
Then, it holds true that
\[
\vert \lambda_n - \lambda \vert \lesssim \exp(-b \, \p),
\]
for some positive constant~$b$ independent of the discretization parameters.
\end{thm}
\begin{proof}
The proof follows the line e.g. of that of~\cite[Theorem 4.3]{VEM_vibration_Kirchhoff}.

Let~$\wn$ be a discrete eigenfunction associated with the discrete eigenvalue~$\lambda_n$, and let $w$ and $\lambda$ be the corresponding exact eigenfunction and eigenvalue, respectively.
Then, one has
\[
\begin{split}
& \a(w-\wn, w-\wn) + \b(w-\wn, w-\wn) - \lambda \c(w-\wn, w-\wn)\\
& = \a(\wn, \wn) + \b(\wn, \wn) - \lambda \c(\wn, \wn)\\
& = \a(\wn, \wn) + \b(\wn, \wn) -\an(\wn, \wn) - \bn(\wn, \wn) + \lambda_n \cn(\wn, \wn) - \lambda\c(\wn, \wn)  . \\
\end{split}
\]
Noting that
\[
\lambda_n \cn(\wn, \wn) - \lambda\c(\wn, \wn)  =\lambda_n \left(  \cn(\wn, \wn) - \c(\wn,\wn)    \right) + (\lambda_n-\lambda) \c(\wn,\wn),
\]
we deduce
\[
\begin{split}
(\lambda_n-\lambda) \c(\wn,\wn)	& = \a(w-\wn, w-\wn) + \b(w-\wn,w-\wn) - \lambda\c(w-\wn,w-\wn)\\
							& \quad + \an(\wn, \wn) - \a(\wn, \wn) + \bn(\wn, \wn) - \b(\wn, \wn)\\
							& \quad - \lambda_n \left[ \cn(\wn, \wn) - \c(\wn, \wn)   \right].\\
\end{split}
\]
The assertion follows by bounding the terms on the right-hand side (with tools similar to those employed in the proof of Theorem~\ref{theorem:Tn:tends:to:T})
and then using the convergence of the eigenfunctions discussed in Remark~\ref{remark:convergence:eigenfunctions}.
\end{proof}
We note that in the standard $\h$-analysis of the convergence of the eigenvalues one typically gets a double rate of convergence.
Since here we are focusing on the $\p$-version on the case of eigenfunctions being the restriction of analytic functions only,
the double rate of convergence is actually hidden within the exponential convergence rate.

\section{Numerical results} \label{section:numerical:results}
In this section, we present a number of numerical experiments that validate the exponential convergence in terms of the degree of accuracy~$\p$ of the discrete eigenvalues to the continuous ones,
whenever the corresponding eigenfunctions are the restriction over the physical domain of analytic functions, see Section~\ref{subsection:p-version}.
Instead, in Section~\ref{subsection:hp-version}, we tackle the instance of eigenfunctions possibly having finite Sobolev regularity, by means of the $\h\p$-version of the method,
and we show that the discrete eigenvalues converge to the continuous ones exponentially in terms of the cubic root of the number of degrees of freedom.


In particular, we are interested in the convergence rate of the normalized error
\begin{equation} \label{quantity:of:interest}
\frac{\vert \lambda - \lambda_ n \vert}{\vert \lambda \vert},
\end{equation}
where~$\lambda$ is a continuous eigenvalues in~\eqref{continuous:problem:weak} and~$\lambda_n$ is a discrete eigenvalue in~\eqref{VEM:eigenvalue} associated with~$\lambda$.
\medskip

In the forthcoming numerical experiments, we could employ the stabilizations introduced in Section~\ref{subsubsection:stabilizations}.
The reason why we picked~$\SEo$ and~$\SEz$ in~\eqref{explicit:stabilization:SE1} and~\eqref{explicit:stabilization:S0}, respectively,
is that we can prove explicit bounds in terms of the ``polynomial'' degree~$\p$ on the parameters $\alpha_*(\p)$, $\alpha^*(\p)$, $\beta_*(\p)$, and $\beta^*$.
A possible effective alternative to $\SEo$ is provided by  the so-called diagonal-recipe stabilizations, defined as
\begin{equation} \label{D-recipe:stabilizations}
\begin{split}
& \widetilde S^\E_1 (\varphi_j, \varphi_i) = \max (1, \aE(\Pinabla \varphi_i, \Pinabla \varphi_j)) \delta_{i,j},\\
\end{split}
\end{equation}
where we recall that $\{ \varphi_i  \}_{i=1}^{\dim(\VnE)}$ is the canonical basis of $\VnE$, and where~$\delta_{i,j}$ denotes the Kronecker delta.

Such stabilization was introduced in~\cite{VEM3Dbasic} and its performance was investigated in~\cite{fetishVEM} and~\cite{fetishVEM3D} in the approximation of a 2D and a 3D Poisson problem, respectively.
If compared to other stabilizing bilinear forms, it entails more robust performance of the method for high ``polynomial'' degree and in presence of distorted or with bad aspect ratio elements.
For this reason, we will employ~$\widetilde S^\E_1$ in~\eqref{D-recipe:stabilizations} and $\SEz$ in~\eqref{explicit:stabilization:S0} as stabilizations for the method.
\medskip

For what concerns the choice of the polynomial basis~$\{\malpha\}_{\alpha=1}^{\pi_{\p-2}}$ dual to the internal moments~\eqref{internal:moments}, we fix an $L^2(\E)$ orthonormal basis elementwise.
In fact, as analyzed in~\cite{fetishVEM}, this choice is particularly effective when the ``polynomial'' degree of the method is high and when the elements are distorted and/or have a bad aspect ratio.
Such a basis is constructed by a stable piecewise $L^2$ orthonormalization process applied to the basis of monomials, elementwise shifted with respect to the barycenter of the element.


\subsection{$\p$-version: the case of analytic eigenfunctions} \label{subsection:p-version}
In this section, we fix our attention to the performance of the $\p$-version of the method in the case of analytic eigenfunctions.

\paragraph*{\texttt{Test case 1}: Laplace on square domain}
As a first test case, we consider as a physical domain the unit square $\Omega_1=(0,1)^2$, partitioned into sequences of Voronoi meshes, see e.g.~\cite{spatialtesselations}.

We aim to approximate the eigenvalues of the Laplace operator on~$\Omega_1$, i.e., we pick $\K = \mathbb {Id}$ (that is, the identity matrix) and $\V = 0$ in~\eqref{continuous:problem:strong}.
The eigenvalues are explicitly known and are given by $\lambda_{k_1,k_2}= (k_1^2 + k_2^2) \pi^2$ for all~$k_1$ and~$k_2$ in~$\mathbb N$ such that~$k_1+k_2\ne 0$.

From Theorem~\ref{theorem:convergence:eigenvalues} and from Remark~\ref{remark:convergence:eigenfunctions}, since all the eigenfunctions are the restriction of analytic functions over~$\mathbb R^2$,
as they have the form $\sin(k_1 \pi x) \sin(k_2 \pi y)$ for all~$k_1$ and~$k_2$ in~$\mathbb N$ such that~$k_1+k_2\ne 0$,
the discrete eigenvalues converge exponentially to the continuous ones in terms of~$\p$,
and therefore in terms of the square root of the number of degrees of freedom, to the continuous ones.

In Figure~\ref{figure:testcase1}, we plot the error~\eqref{quantity:of:interest} against the square root of the number of degrees of freedom for the first four eigenvalues.
We employ the $\h$-version for $\p=1$, 2, and~$3$, and the $\p$-version of the method.
For the $\p$-version, we employ a Voronoi mesh corresponding with the coarsest mesh of the $\h$-version.
On the $x$-axis, we plot the square root of the number of degrees of freedom.
\begin{figure}  [h]
\centering
\includegraphics [angle=0, width=0.45\textwidth]{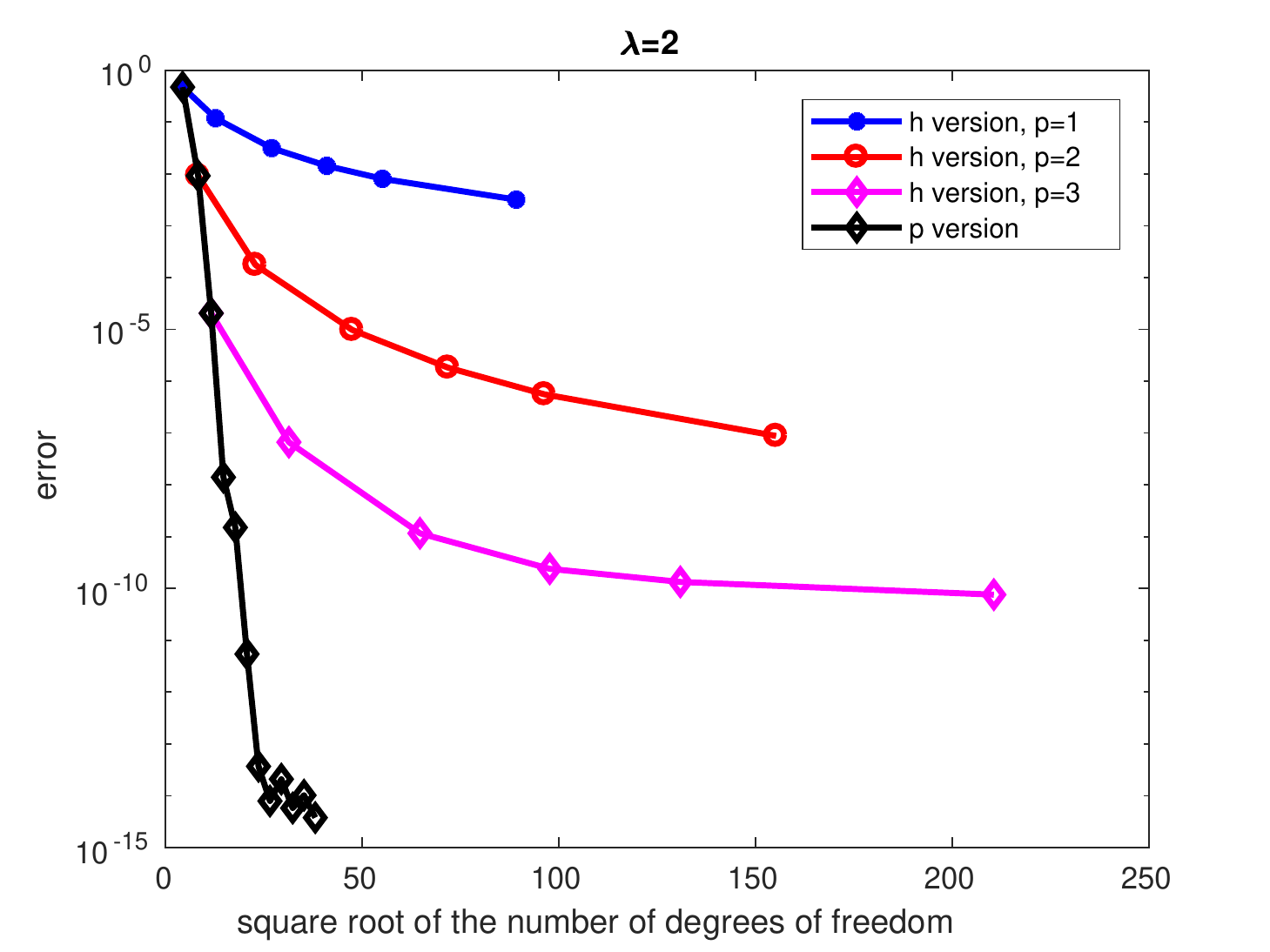}
\includegraphics [angle=0, width=0.45\textwidth]{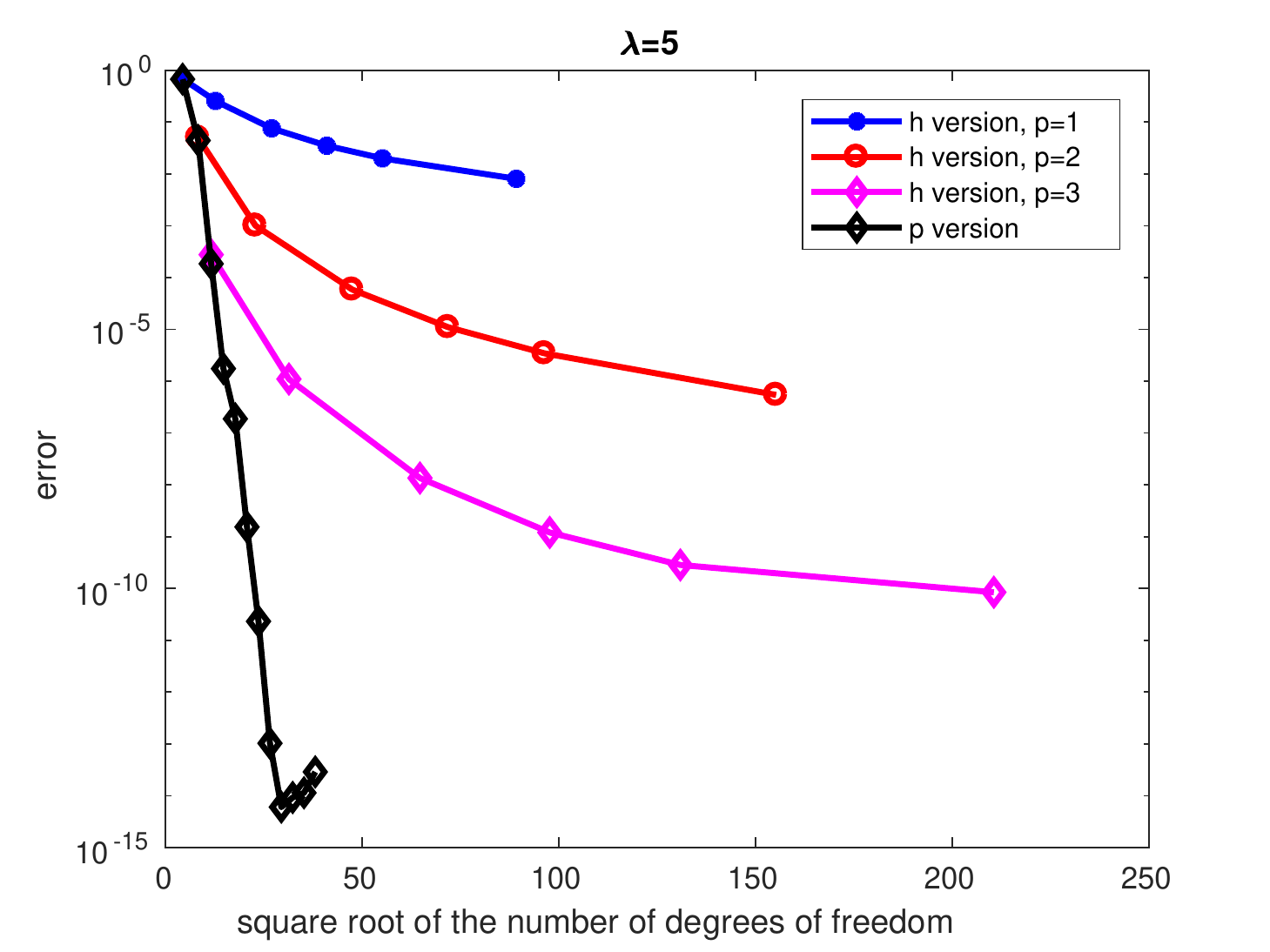}
\includegraphics [angle=0, width=0.45\textwidth]{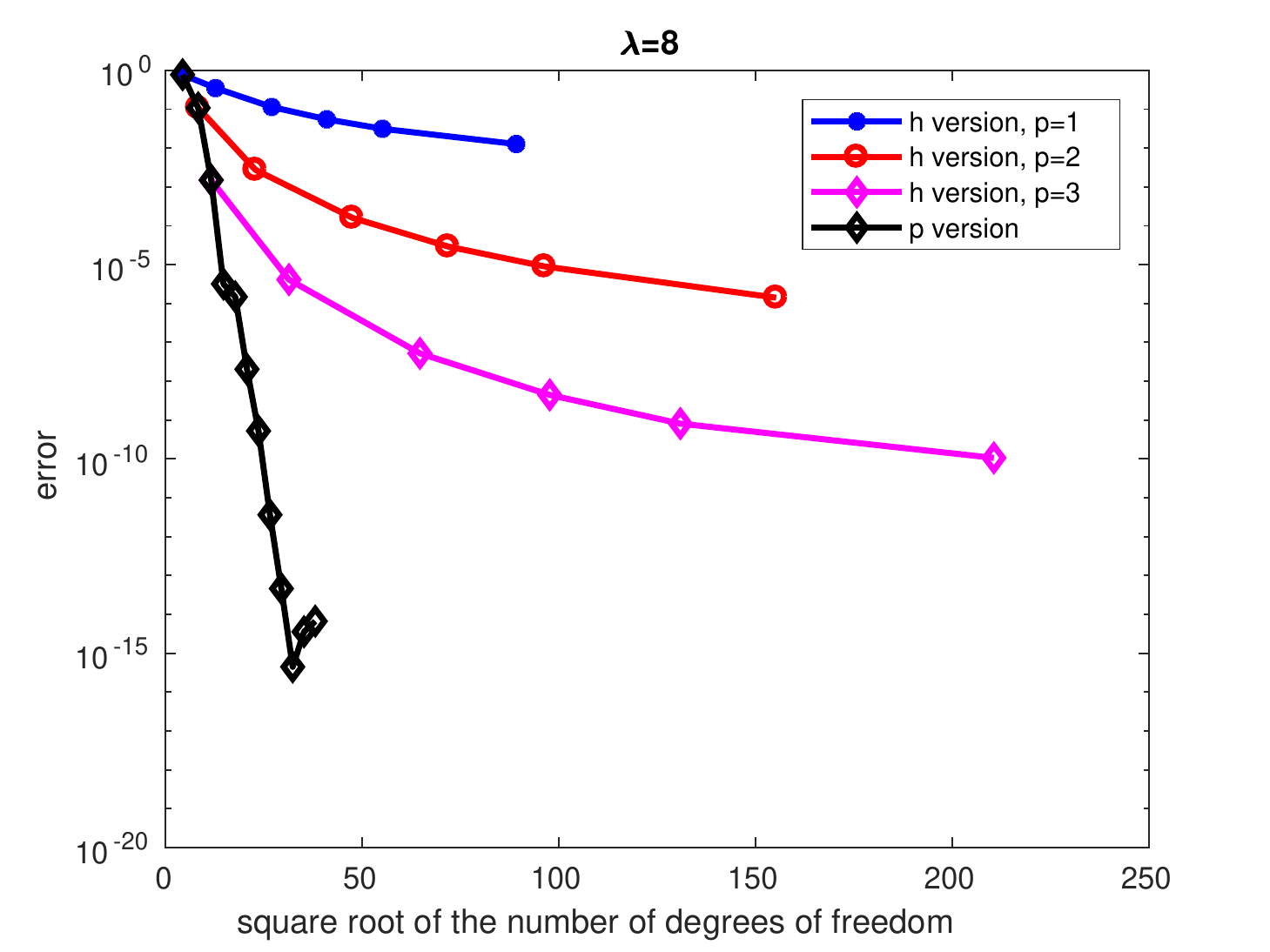}
\includegraphics [angle=0, width=0.45\textwidth]{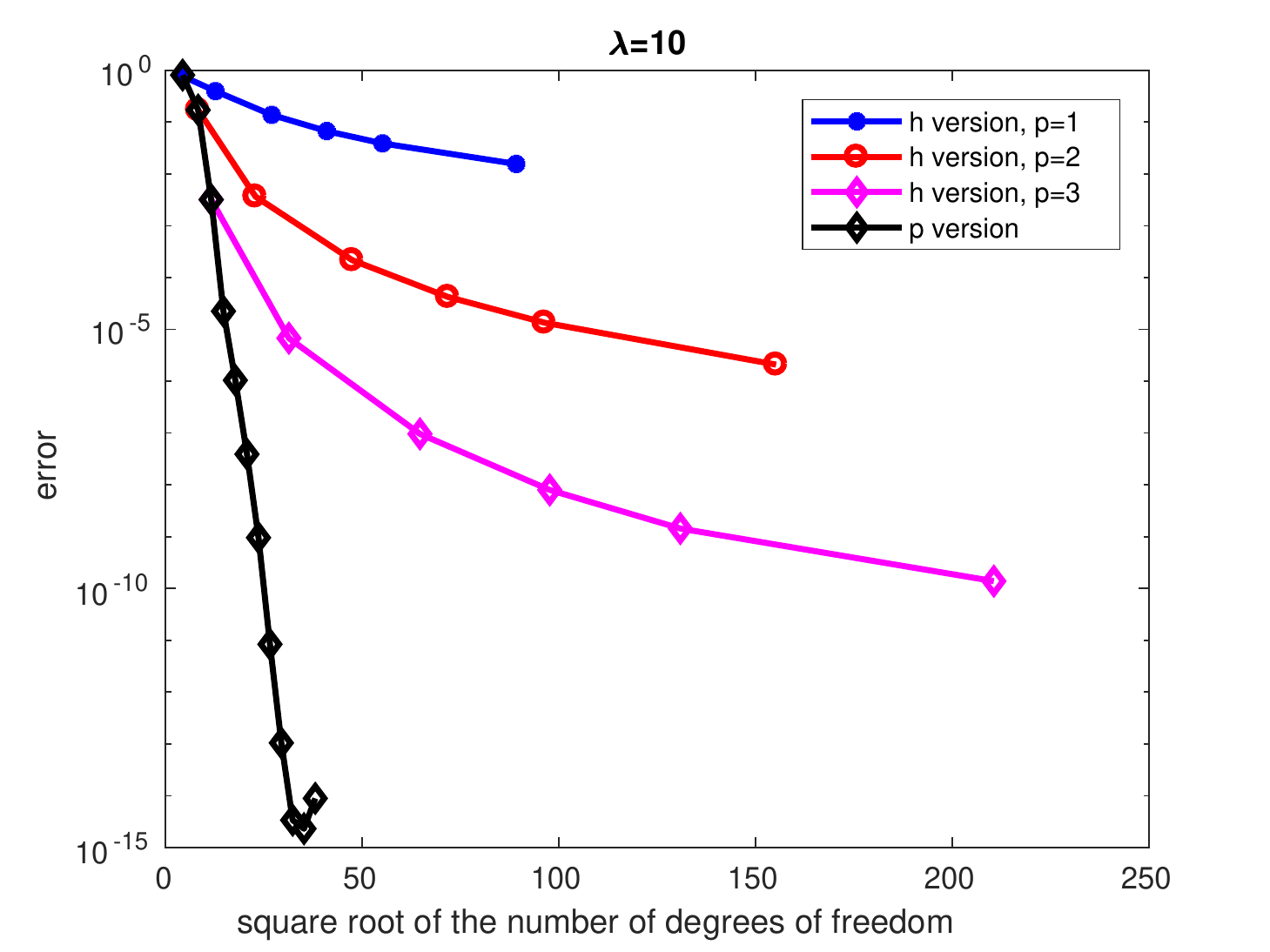}
\caption{Convergence of the error~\eqref{quantity:of:interest} for
    the first four distinct Dirichlet eigenvalues of the Laplace
    operator on the square domain~$\Omega_1$ employing the
    $\h$-version with $\p=1$, $2$, and~$3$ and the $\p$-version of the
    method.      
    On the $x$-axis, we plot the square root of the number of degrees
    of freedom.  
    The stabilizations~$\widetilde S_1^\E$ and~$\SEz$ are defined
    in~\eqref{D-recipe:stabilizations} and~\eqref{explicit:stabilization:S0}, respectively.  
    The polynomial basis dual to the internal moments~\eqref{internal:moments} is $L^2$ orthonormal elementwise.  
    For both the $\h$- and the $\p$-versions, we employ Voronoi
    meshes.  
     The error curves plotted in the figure refer to the approximation of the problem eigenvalues as follows:
      \textit{top-left panel:} first eigenvalue; 
      \textit{top-right panel:} second eigenvalue;
      \textit{bottom-left panel:} third eigenvalue;
      \textit{bottom-right panel:} fourth eigenvalue.
 }
\label{figure:testcase1}
\end{figure}

\paragraph*{\texttt{Test case 2}: quantum harmonic oscillator on a square}
Another interesting test case with analytic eigenfunctions is provided
by the quantum harmonic oscillator~\cite{Heisenberg:1925,Griffiths:1995}, that is, when one
fixes~$\K=0.5\mathbb I$ (that is, again, the identity matrix) and
$\V(x,y) = 0.5 (x^2+y^2)$ in~\eqref{continuous:problem:strong}.

The eigenfunctions on~$\mathbb R^2$ are given by the product of the
Gaussian bell $\exp(-(x^2+y^2))$ with a tensor product of Hermite
polynomials.
The eigenvalues are all the natural numbers; every eigenvalue $n\in
\mathbb N$ has multiplicity precisely equal to~$n$.

The eigenfunctions of the quantum harmonic oscillator have not zero boundary conditions on bounded domains; however, they decrease rapidly to zero as~$x$ and~$y$ tend to infinity.
For this reason, we consider as a physical domain, the (sufficiently wide) square $\Omega_2=(-10,10)^2$, and we impose zero boundary conditions, assuming that the resulting eigenfunctions and eigenvalues
are practically given by those in the unbounded domain $\mathbb R^2$.
Again, we compare the performance of the $\h$-version with $\p=1$, $2$, and~$3$ with the $\p$-version of the method on Voronoi meshes.
For the $\p$-version, we employ a Voronoi mesh corresponding with the second coarsest mesh of the $\h$-version.
On the $x$-axis, we consider the square root of the number of degrees of freedom, which we recall is a consequence of Theorem~\ref{theorem:convergence:eigenvalues}.
\begin{figure}  [h]
\centering
\includegraphics [angle=0, width=0.45\textwidth]{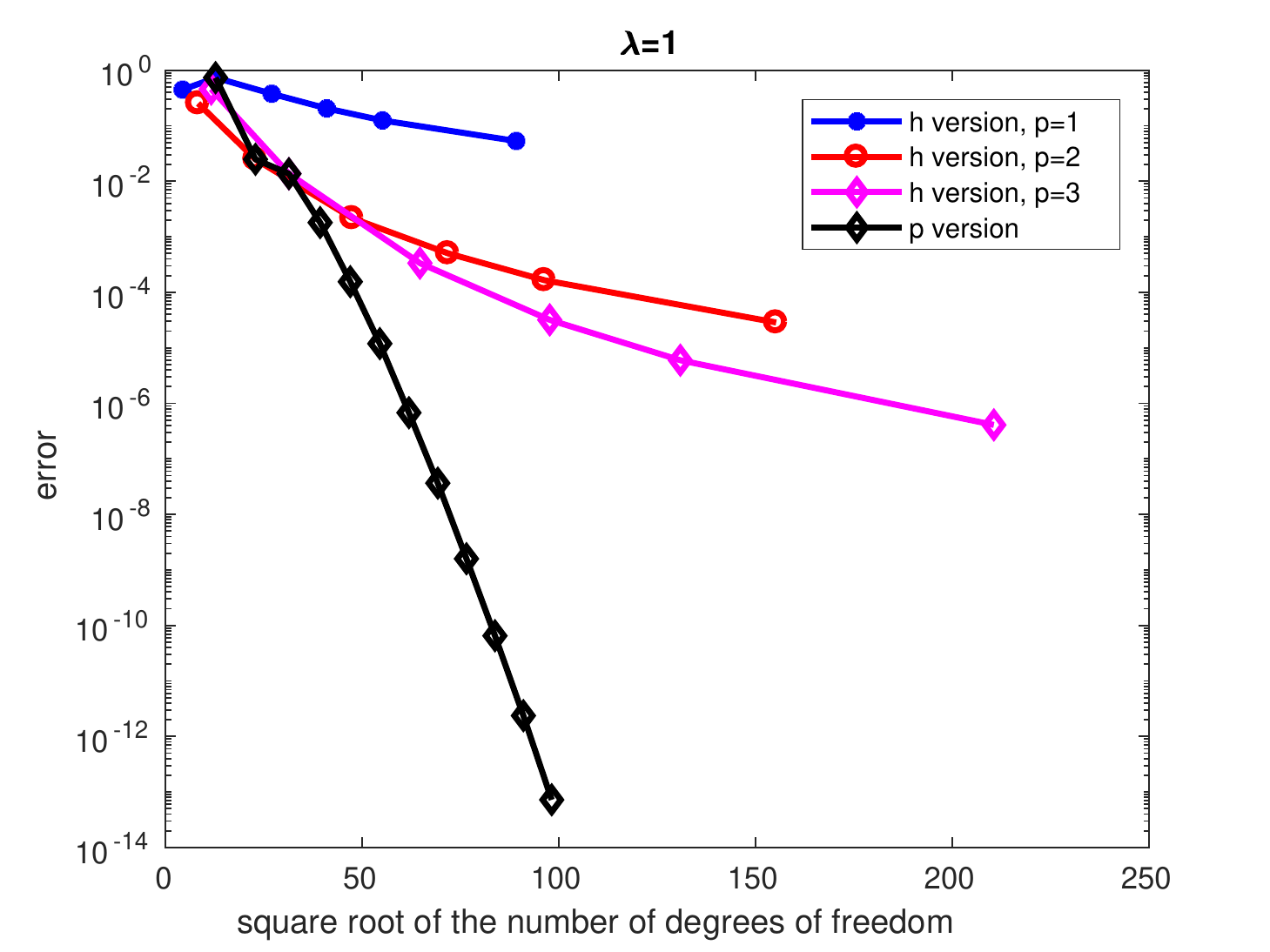}
\includegraphics [angle=0, width=0.45\textwidth]{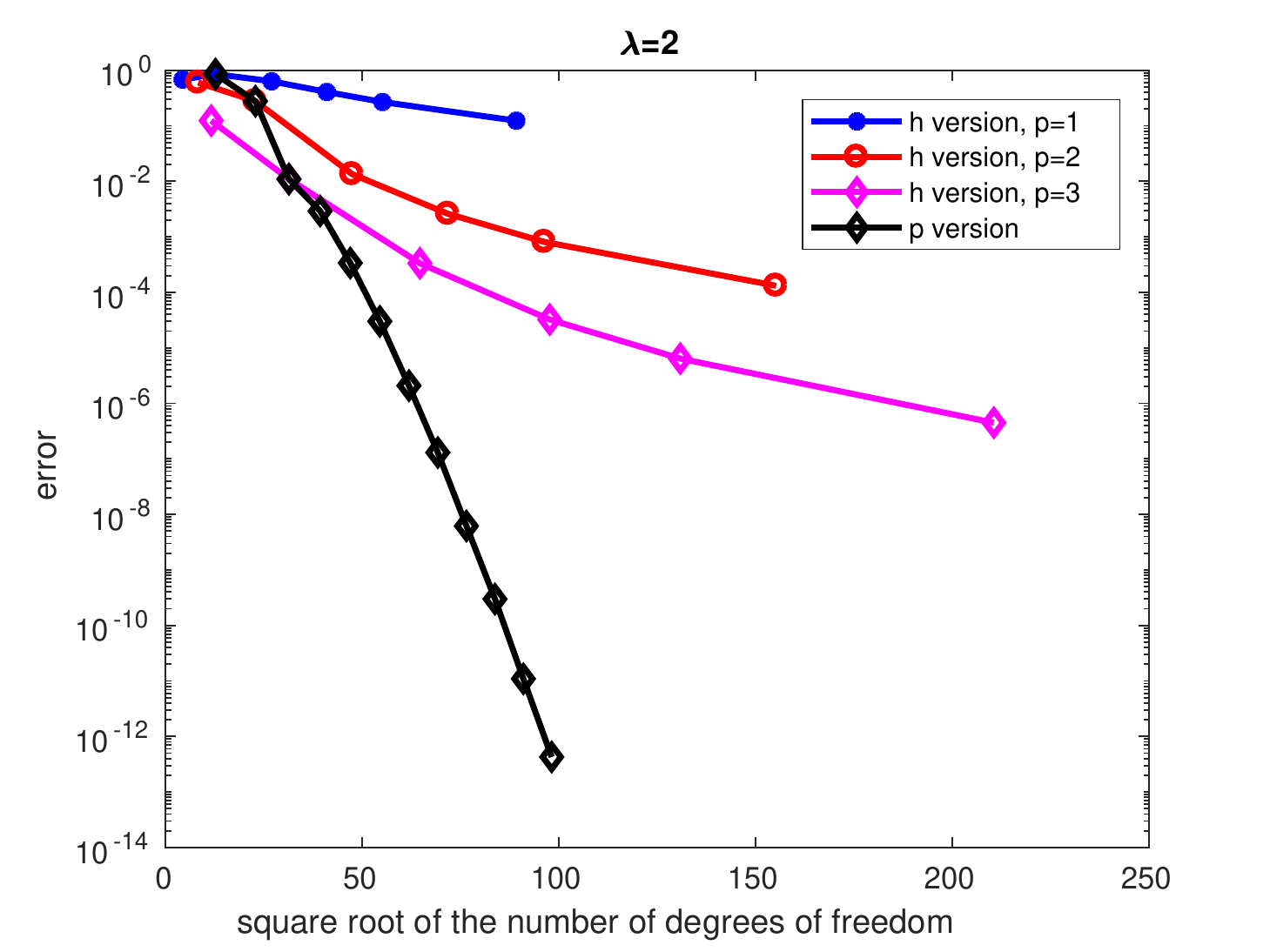}
\includegraphics [angle=0, width=0.45\textwidth]{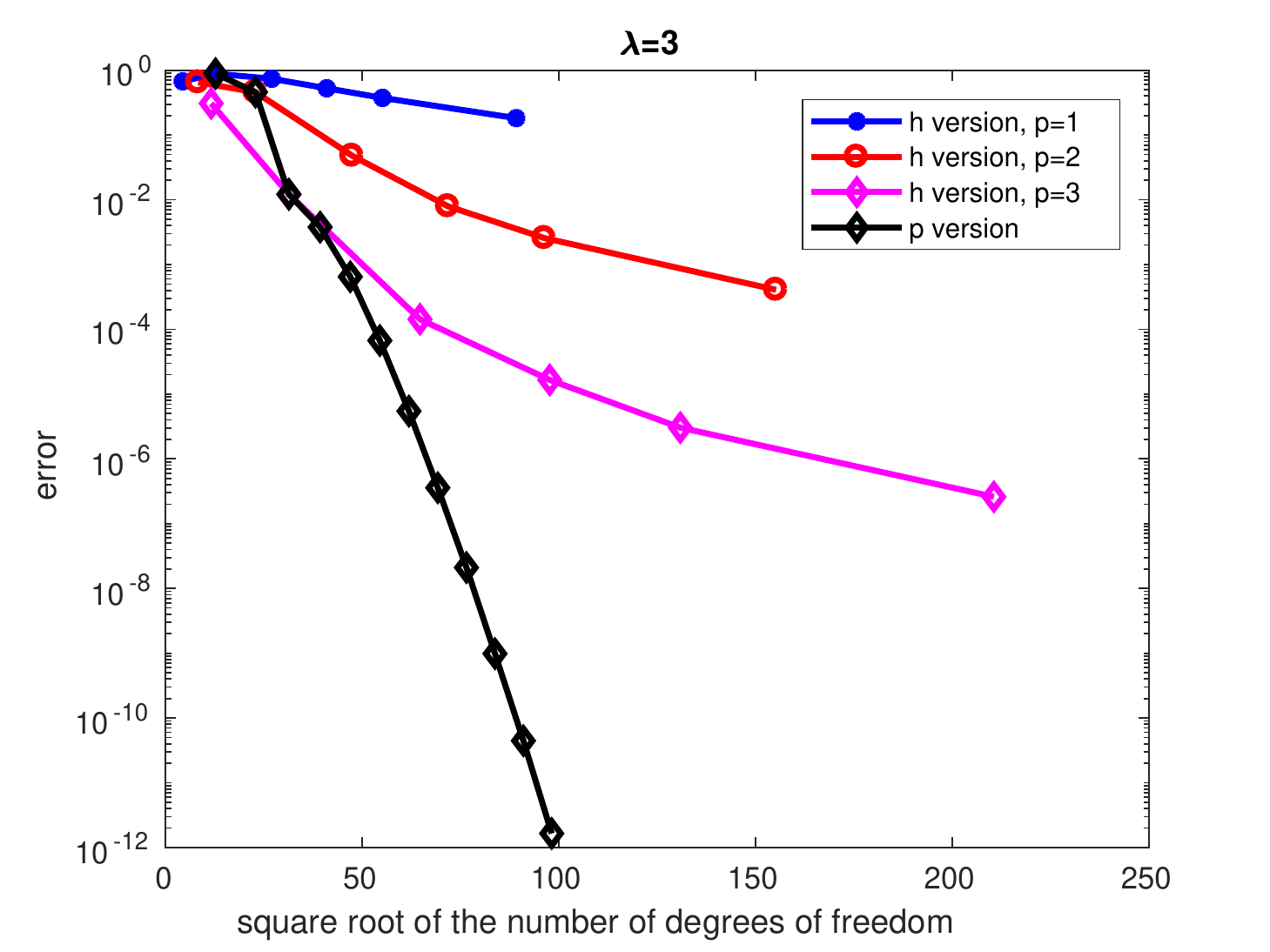}
\includegraphics [angle=0, width=0.45\textwidth]{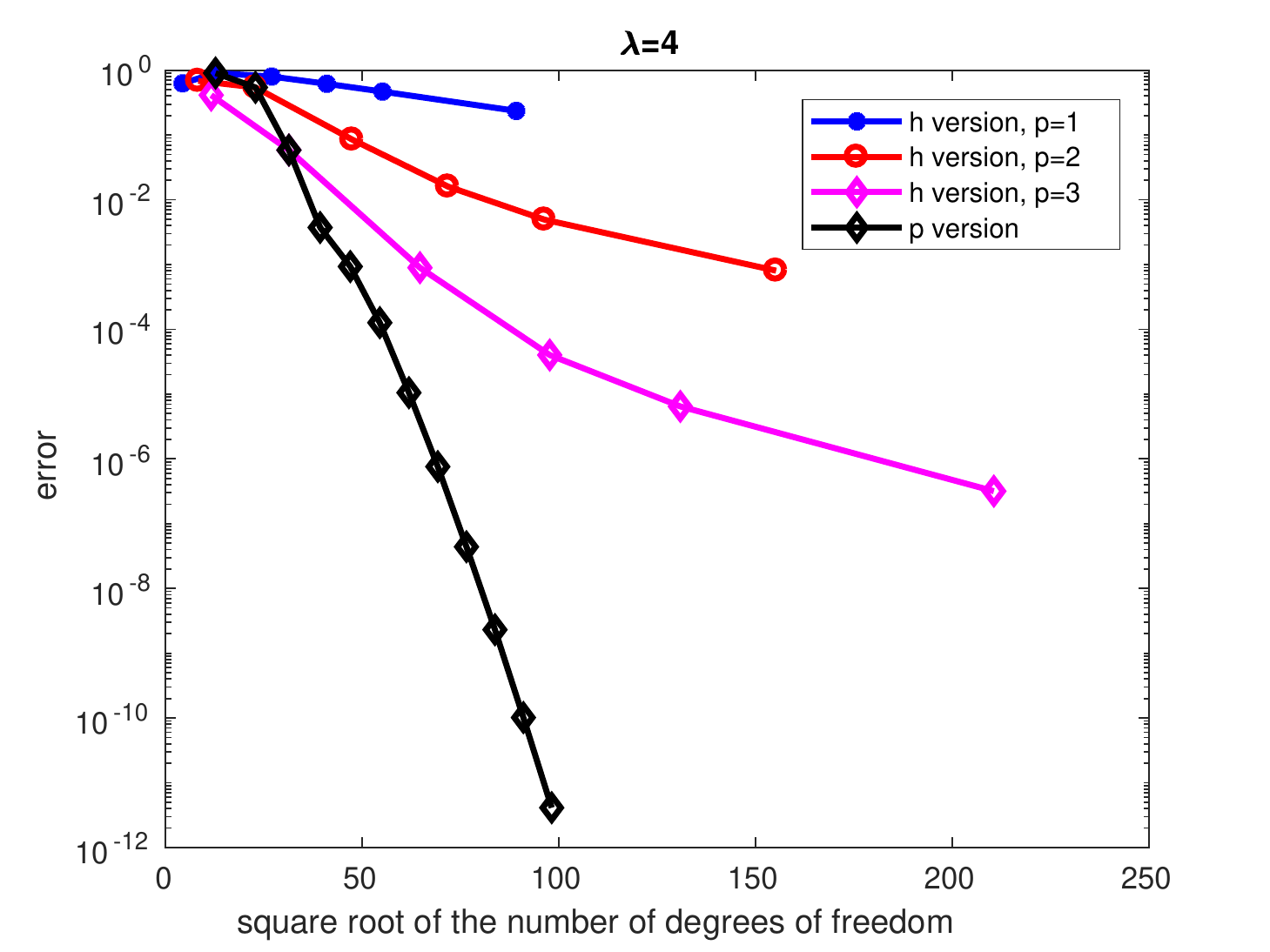}
\caption{Convergence of the error~\eqref{quantity:of:interest} for the first four distinct Dirichlet eigenvalues of the quantum harmonic oscillator operator on the square domain~$\Omega_2$
employing the $\h$-version with $\p=1$, $2$, and~$3$ and the $\p$-version of the method. On the $x$-axis, we plot the square root of the number of degrees of freedom.
The stabilizations~$\widetilde S_1^\E$ and~$\SEz$ are defined in~\eqref{D-recipe:stabilizations} and~\eqref{explicit:stabilization:S0}, respectively.
The polynomial basis dual to the internal moments~\eqref{internal:moments} is $L^2$ orthonormal elementwise.
For both the $\h$- and the $\p$-versions, we employ Voronoi meshes.
The error curves plotted in the figure refer to the
      approximation of the problem eigenvalues as follows:
      \textit{top-left panel:} first eigenvalue; 
      \textit{top-right panel:} second eigenvalue;
      \textit{bottom-left panel:} third eigenvalue;
      \textit{bottom-right panel:} fourth eigenvalue.}
\label{figure:testcase2}
\end{figure}

In Figure~\ref{figure:testcase2}, one can appreciate the exponential decay of the error in terms of the square root of the number of degrees of freedom.

\subsection{$\h\p$-version: the case of singular eigenfunctions} \label{subsection:hp-version}
This section is devoted to the approximation of eigenfunctions that are not analytic, but rather could present some singularities.
For the sake of simplicity, we consider test cases where the singularities are concentrated at isolated points only;
for instance, we avoid the more technical case of edge singularities, where anisotropic mesh refinements could come into play.

In Section~\ref{section:convergence:analysis}, we analyzed the convergence rate of the method and we underlined the suboptimality of the $\p$-version for nonanalytic eigenfunctions, see Theorem~\ref{theorem:convergence:source:problem} and the comments below.
In particular, the bound in Theorem~\ref{theorem:convergence:source:problem}, together with the stray behavior in terms of~$\p$ of the stability constants~$\alpha_*(\p)$ and~$\alpha^*(\p)$,
could lead in principle also to a divergent method for eigenfunctions with sufficiently strong singularities.

Thus, we resort to the $\h\p$-version of the method, which will be described in the forthcoming Section~\ref{subsubsection:hp:VES}.
Such a strategy follows the line of the construction in~\cite{SchwabpandhpFEM, hpVEMcorner, BabuGuo_hpFEM};
by a proper combination of local mesh refinements and increasing the number of the ``polynomial'' degree over the polygonal decomposition in a nonuniform way,
is possible to recover exponential convergence of the method in terms of the cubic root of the number of degrees of freedom.

A couple of test cases dealing with the $\h\p$-version of the method will be presented in Section~\ref{subsubsection:experiments:hp}.

\subsubsection{$\h\p$-virtual element spaces} \label{subsubsection:hp:VES}
In this section, we describe the structure of  $\h\p$-virtual element space, which will be instrumental for the approximation of eigenfunctions with finite Sobolev regularity, see Section~\ref{subsubsection:experiments:hp}.

The idea behind the $\h\p$-refinements is that geometric mesh refinements are performed on the elements where the solution of the target problem is singular,
whereas $\p$-refinements are performed on the elements where the solution is analytic.
Henceforth, we assume for the sake of simplicity that the eigenfunctions are analytic everywhere in~$\Omega$, but at a set~$\M$ of~$M$ points, lying either in the interior of $\Omega$ or on its boundary.

As a first step in this construction, we introduce the concept of layers associated with~$\M$ of a sequence of meshes~$\{\taun\}_{n\in \mathbb N}$.
Given $n\in \mathbb N$, we assume that the mesh~$\taun$ consists of~$n+1$ layers, where the~$0$-th layer~$L_0^n$ consists of all the elements abutting the points in~$\M$.
The other layers are defined by induction as
\begin{equation*}
L_j^n = \left\{  \E_1 \in \taun \mid \overline {\E_1} \cap \overline {\E_2}\ne \emptyset \text{ for some }\E_2\in L^n_{ j -1},\, \E_1 \notin \cup_{i=0}^{j-1} L_i^n    \right\} \quad \forall j=1,\dots,n.
\end{equation*}
Next, we introduce the concept of geometrically graded meshes~$\{\taun\}_{n\in \mathbb N}$. For all $n\in \mathbb N$, there exists a (grading) parameter $\sigma\in (0,1)$ such that
\[
\hE \approx \sigma ^{n-j} \quad \text{in } \E \in L^n_j.
\]
Therefore, geometrically graded meshes are characterized by very ``tiny'' elements abutting the ``singular'' points in~$\M$, and by elements increasing their size geometrically, when increasing the index of the layer they belong to.
Roughly speaking,  the ``tiny'' elements guarantee good approximation properties where the exact solution is singular.
In Figures~\ref{figure:L-shaped} and~\ref{figure:hp:checkerboard}, we have depicted the first three meshes of two sequences of geometric graded meshes; we have highlighted different layers in different colors.

The local ``polynomial'' degrees are distributed in a nonuniform way. More precisely, we fix a parameter $\mu \in \mathbb N$ and define a vector $\pbold \in \mathbb N^{\card (\taun)}$, such that
\begin{equation} \label{hp:degrees:of:accuracy}
\pbold _\ell = \mu (j+1) \quad \text{if} \quad \E_\ell \in L_j^n\quad \forall \ell=1,\dots, \card (\taun).
\end{equation}
The idea behind this choice is that the local ``polynomial'' degree of the method on~$\VnE$ increases linearly when increasing the layer index. This is sufficient in standard $\h\p$-methods \cite{SchwabpandhpFEM, hpVEMcorner, BabuGuo_hpFEM}
to get exponential convergence in terms of the cubic root of the number of degrees of freedom.

Next, we construct another vector $\pboldEpsilon \in \mathbb N^{\card(\mathcal E_n)}$, whose entries are defined by
\begin{equation} \label{maximum:rule}
\pboldEpsilon{}_{\widetilde \ell} =
\begin{cases}
\max(\pbold_{\ell_1},\pbold_{\ell_2})	& \text{if } \e_{\widetilde \ell} \in \mathcal E_n^I \text{ and } \e_{\widetilde \ell} \subseteq \partial \E_{\ell_1} \cap \partial \E_{\ell_2}\\
\pbold_\ell 					& \text{if } \e_{\widetilde \ell} \in \mathcal E_n^B \text{ and } \e_{\widetilde \ell}  \subset \partial \E_\ell,
\end{cases}
\end{equation}
where we recall that~$\mathcal E_n^B$ and~$\mathcal E_n^I$ denote the set of boundary and internal edges of~$\taun$, respectively.

The $\h\p$-virtual element spaces are consequently defined by considering the Laplacian in  the space of polynomials of degree~$\p_\ell$ on the element~$\E_\ell$, for all $\ell=1,\dots,\card(\taun)$,
fixing piecewise continuous polynomial Dirichlet traces over the edges with degree chosen accordingly to the maximum rule~\eqref{maximum:rule},
and then imposing the enhancing constraints locally as in~\eqref{local:enhanced:space}.

Following the lines of~\cite[Section 5]{hpVEMcorner}, it is possible to prove that, employing the $\h\p$-spaces, the error $\vert u - \un \vert_{1,\Omega}$
converges exponentially in terms of the cubic root of the number of degrees of freedom, also if the estimates in Theorem~\ref{theorem:convergence:source:problem} are severely suboptimal in~$\p$.
The important point when trying to recover exponential convergence is that the suboptimal factor haunting the $\h\p$-version of the VEM grows at most algebraically in~$\p$.

Since the matter is technical and follows broadly by combining the techniques of~\cite{hpVEMcorner} and the results in Section~\ref{section:convergence:analysis},
we limit ourselves to present here the numerical results.
As an interesting side remark, we underline that so far the construction of the $\h\p$-strategy was based on a priori knowledge of the singular behavior of the eigenfunctions.
One could also build $\h\p$-spaces in an adaptive fashion, employing for instance residual-based a posteriori error analysis, as done in~\cite{hpVEMapos}.

\subsubsection{Numerical experiments} \label{subsubsection:experiments:hp}
In this section, we present numerical experiments on a couple of test cases where the eigenfunctions have (possibly) finite Sobolev regularity, at some isolated points.

In particular, we compare the performance of the $\h$- and of the $\h\p$-versions.
In the first test case, the singular behavior is due to the shape of the physical domain, which is assumed to be L-shaped.
In the second, it is instead due to the discontinuity of the diffusivity tensor~$\K$; namely, we will consider the so-called \emph{checkerboard benchmark}, see e.g.~\cite{CiarletJamelotKpadonou}.

The $\h$-version is performed always employing uniform Cartesian meshes. In the checkerboard case, the Cartesian meshes are assumed to be conforming with respect to the discontinuities of the diffusivity tensor.
We also underline that the construction of $\h\p$-spaces benefit from the possibility of using polygonal meshes, see Figures~\ref{figure:L-shaped} and~\ref{figure:hp:checkerboard}.
\paragraph*{\texttt{Test case 3}: Laplace on L-shaped domain}
In this test case, we fix as a physical domain, the L-shaped domain $\Omega_3 = (-1,1)^2 \setminus (-1,0]^2$. We look for eigenvalues of the Laplace operator; thus, we set $\K=\mathbb I$ (that is, the identity matrix) and $\V=0$ in~\eqref{continuous:problem:strong}.
For what concerns the $\h\p$-version, we consider a distribution of the ``polynomial'' degrees as in~\eqref{hp:degrees:of:accuracy}, picking~$\mu=1$.
We consider here homogeneous Neumann boundary conditions; note that the method when imposing Neumann boundary conditions is defined similarly to~\eqref{VEM:eigenvalue}, see e.g.~\cite[Remark 2.2]{pVEMmultigrid}.

The geometrically graded meshes are built by taking as a grading parameter $\sigma=0.5$, the geometric refinement is towards the re-entrant corner.
The first three meshes, together with the corresponding distribution of the ``polynomial'' degrees, are depicted in Figure~\ref{figure:L-shaped}.
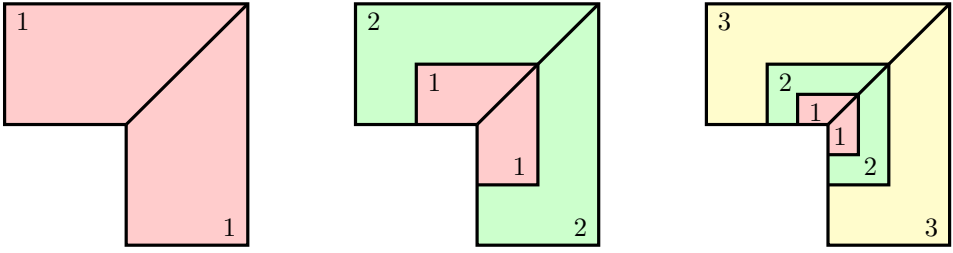
\begin{figure}  [h]
\begin{center}
\begin{minipage}{0.30\textwidth}
\begin{tikzpicture}[scale=0.4]
\fill[red, opacity=0.2] (0,0) -- (0,-4) -- (4,-4) -- (4,4) -- (-4,4) -- (-4,0) -- (0,0);
\draw[black, very thick, -] (0,0) -- (0,-4) -- (4,-4) -- (4,4) -- (-4,4) -- (-4,0) -- (0,0);
\draw[black, very thick, -] (0,0) -- (4,4);
\draw(-3.4,3.4) node[black] {{$1$}}; \draw(3.4, -3.4) node[black] {{$1$}};
\end{tikzpicture}
\end{minipage}
\begin{minipage}{0.30\textwidth}
\begin{tikzpicture}[scale=0.4]
\fill[green, opacity=0.2] (-4,0) -- (-2,0)-- (-2,2)-- (2,2)-- (2,-2)-- (0,-2)-- (0,-4)-- (4,-4) -- (4,4) -- (-4,4);
\fill[red, opacity=0.2] (0,-2) -- (2,-2) -- (2,2) -- (-2,2) -- (-2,0) -- (0,0);
\draw[black, very thick, -] (0,0) -- (0,-4) -- (4,-4) -- (4,4) -- (-4,4) -- (-4,0) -- (0,0);
\draw[black, very thick, -] (0,-2) -- (2,-2) -- (2,2) -- (-2,2) -- (-2,0);
\draw[black, very thick, -] (0,0) -- (4,4);
\draw(-3.4,3.4) node[black] {{$2$}}; \draw(3.4, -3.4) node[black] {{$2$}};
\draw(-1.4,1.4) node[black] {{$1$}}; \draw(1.4, -1.4) node[black] {{$1$}};
\end{tikzpicture}
\end{minipage}
\begin{minipage}{0.33\textwidth}
\begin{tikzpicture}[scale=0.4]
\fill[yellow, opacity=0.2] (-4,0) -- (-2,0)-- (-2,2)-- (2,2)-- (2,-2)-- (0,-2)-- (0,-4)-- (4,-4) -- (4,4) -- (-4,4);
\fill[green, opacity=0.2] (-2,0) -- (-1,0)-- (-1,1)-- (1,1)-- (1,-1)-- (0,-1)-- (0,-2)-- (2,-2) -- (2,2) -- (-2,2);
\fill[red, opacity=0.2] (0,-1) -- (1,-1) -- (1,1) -- (-1,1) -- (-1,0) -- (0,0);
\draw[black, very thick, -] (0,0) -- (0,-4) -- (4,-4) -- (4,4) -- (-4,4) -- (-4,0) -- (0,0);
\draw[black, very thick, -] (0,-2) -- (2,-2) -- (2,2) -- (-2,2) -- (-2,0);
\draw[black, very thick, -] (0,-1) -- (1,-1) -- (1,1) -- (-1,1) -- (-1,0);
\draw[black, very thick, -] (0,0) -- (4,4);
\draw(-3.4,3.4) node[black] {{$3$}}; \draw(3.4, -3.4) node[black] {{$3$}};
\draw(-1.4,1.4) node[black] {{$2$}}; \draw(1.4, -1.4) node[black] {{$2$}};
\draw(-.4,.4) node[black] {{$1$}}; \draw(.4, -.4) node[black] {{$1$}};
\end{tikzpicture}
\end{minipage}
\end{center}
\caption{First three meshes~$\mathcal T_0$,~$\mathcal T_1$, and~$\mathcal T_2$, and the distribution of the ``polynomial'' degree for the approximation of eigenfunctions and eigenvalues of the Laplace operator on the L-shaped domain~$\Omega_3$.
The $0$-th, the $1$-st, and the $2$-nd layers are highlighted in \red{red}, \green{green}, and \yellow{yellow} colors, respectively.
The layers are constructed refining only towards the reentrant corner.}
\label{figure:L-shaped}
\end{figure}

In order to test the method, we compare our discrete eigenvalues with those described in Dauge's website~\cite{Dauge_benchmark}.
The numerical results, describing the convergence to the first four distinct eigenvalues
when employing the $\h$- and the $\h\p$-versions of the method on a sequence of uniform Cartesian meshes and on the sequence of graded meshes in Figure~\ref{figure:L-shaped}, are depicted in Figure~\ref{figure:testcase3}.
\begin{figure}  [h]
\centering
\includegraphics [angle=0, width=0.45\textwidth]{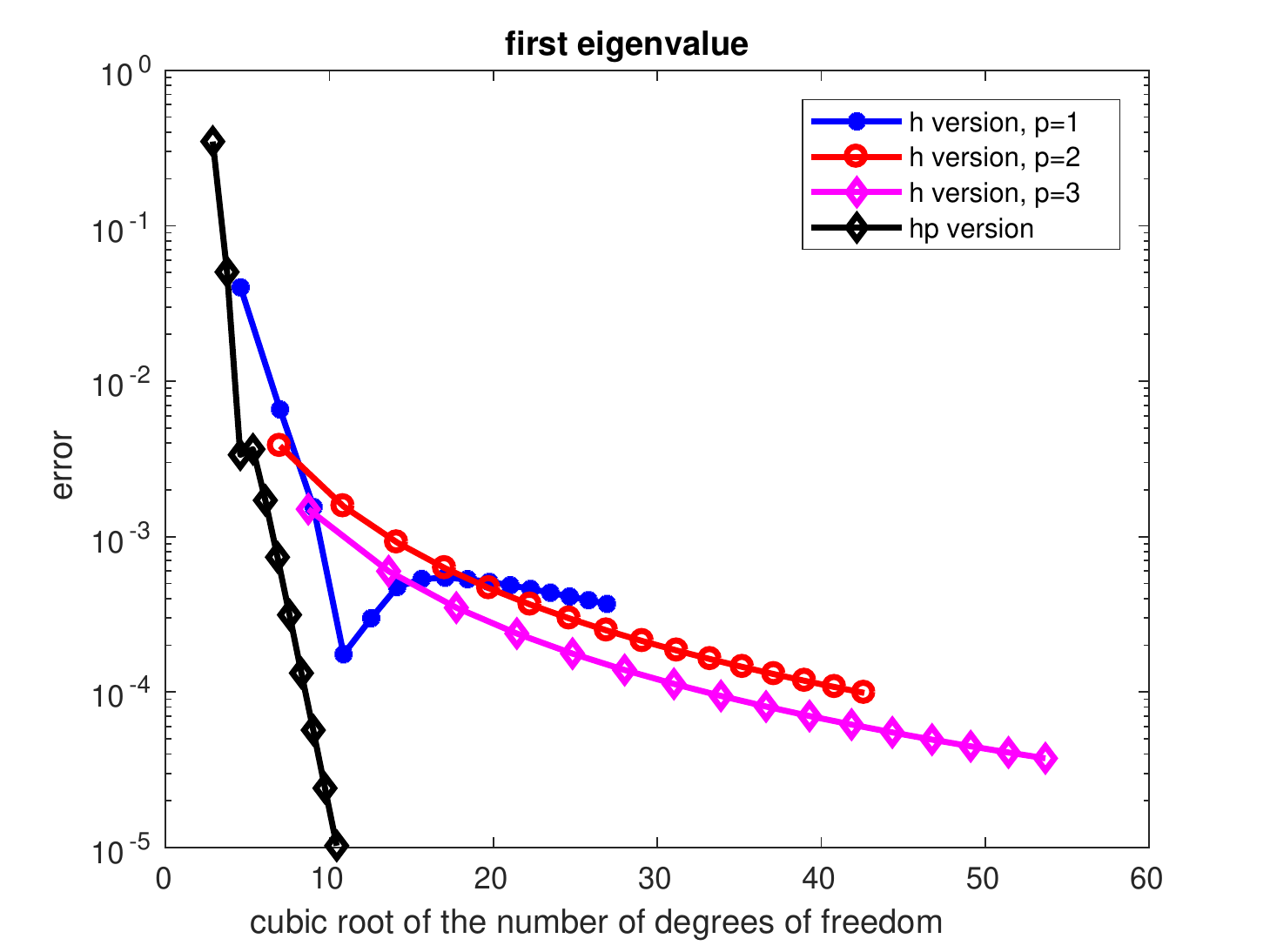}
\includegraphics [angle=0, width=0.45\textwidth]{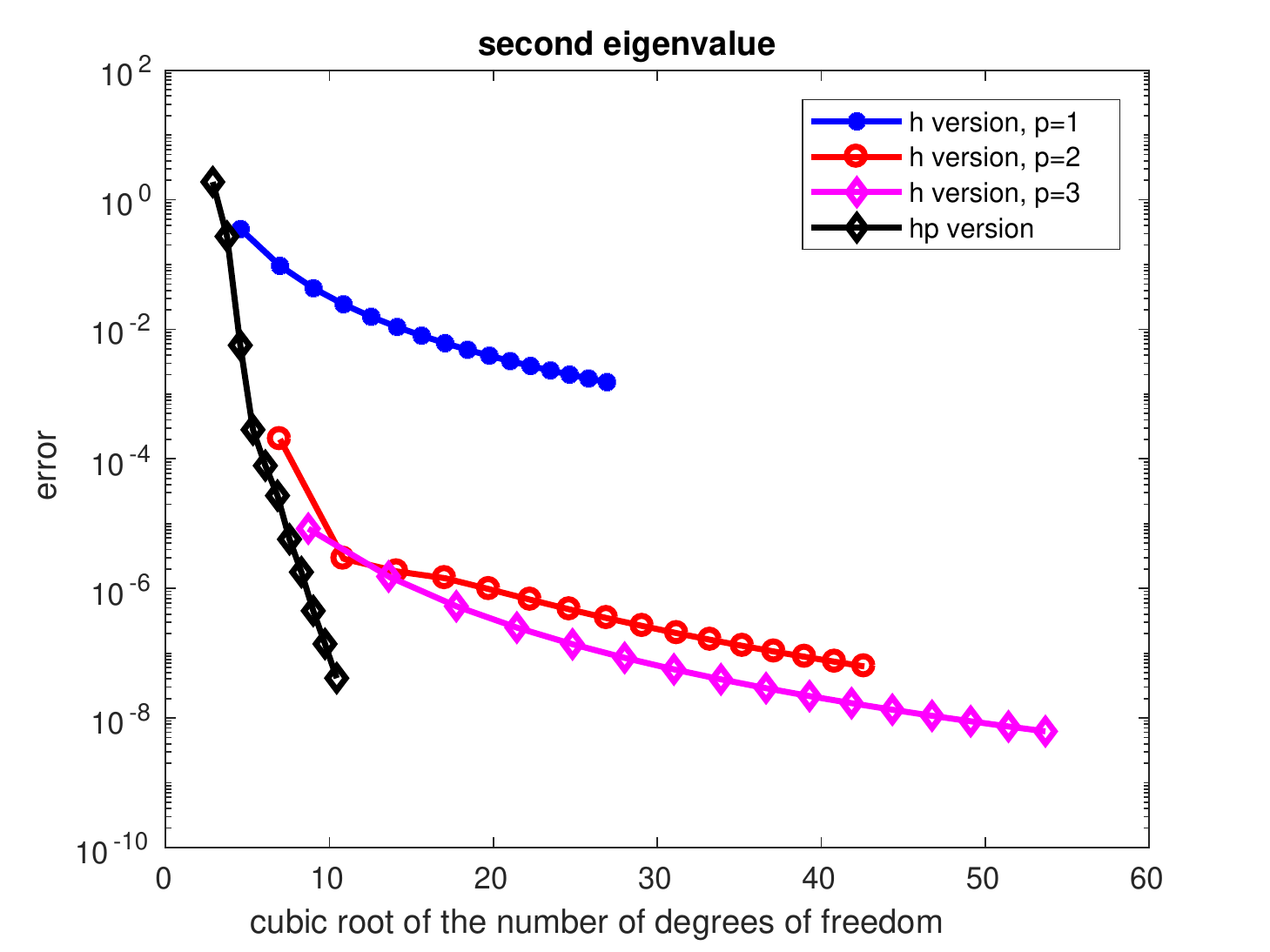}
\includegraphics [angle=0, width=0.45\textwidth]{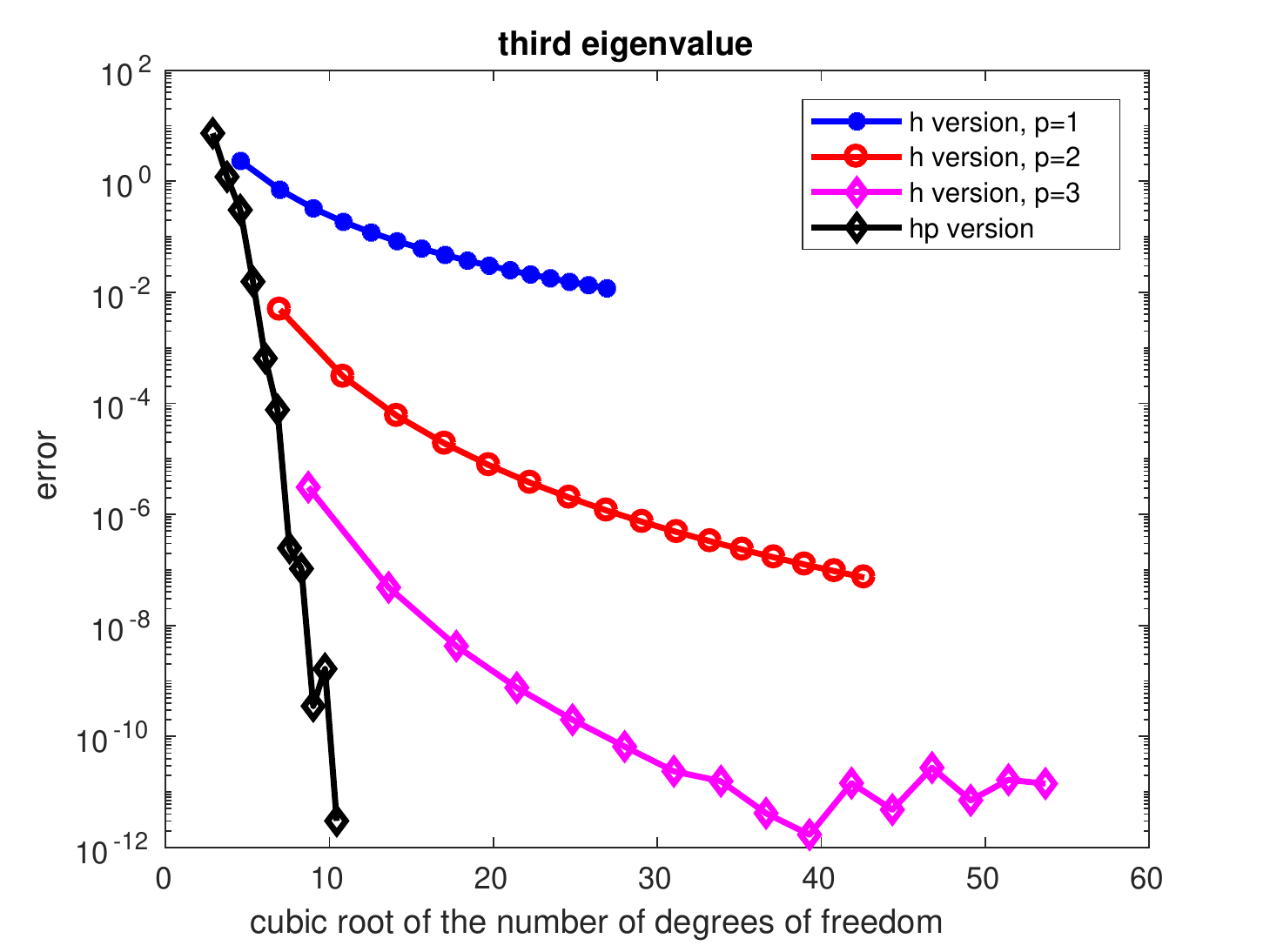}
\includegraphics [angle=0, width=0.45\textwidth]{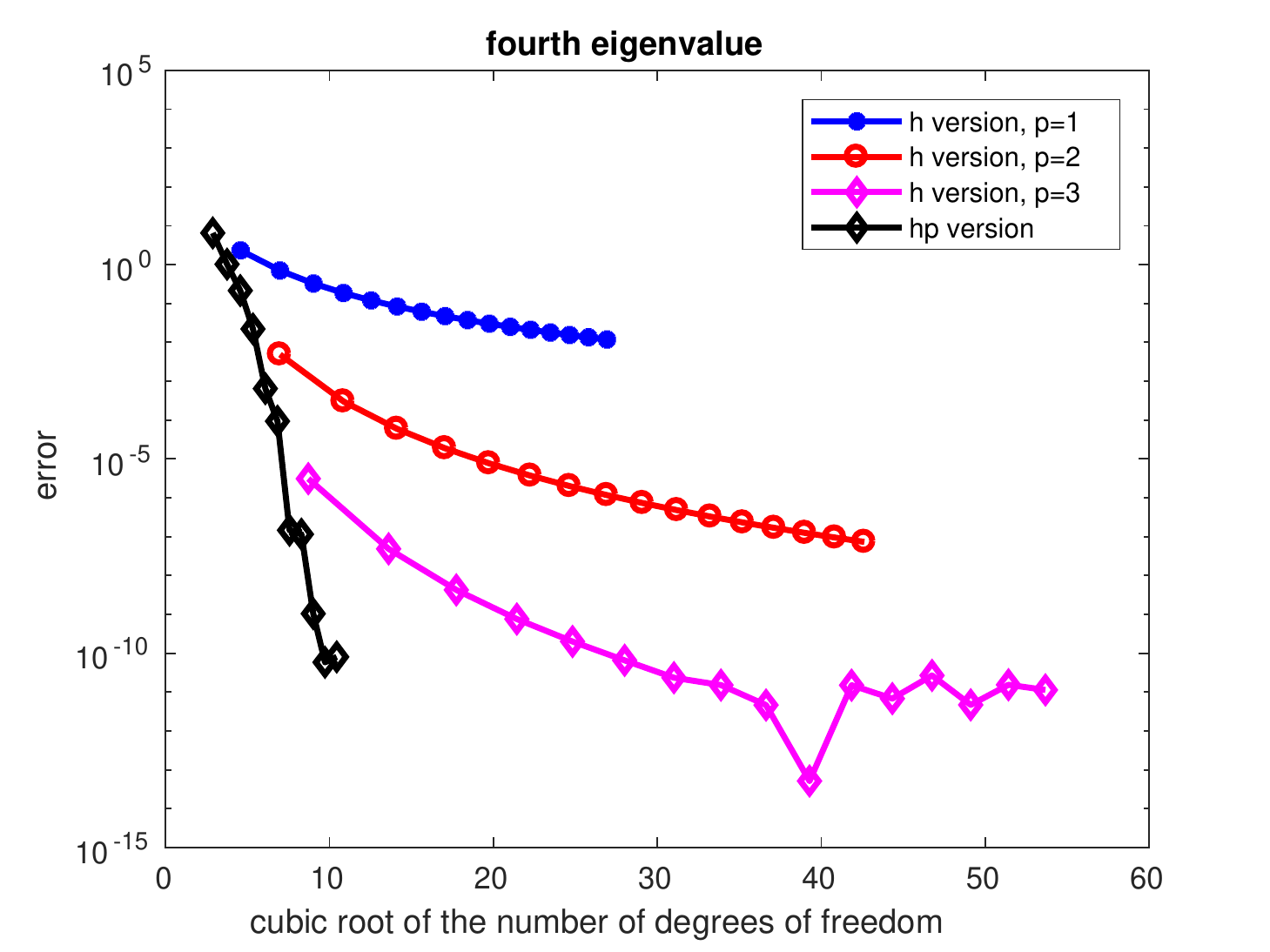}
\caption{Convergence of the error~\eqref{quantity:of:interest} for the first four distinct Neumann eigenvalues of the Laplace operator on the L-shaped domain~$\Omega_3$
employing the $\h$-version with $\p=1$, $2$, and~$3$ and the $\h\p$-version of the method. On the $x$-axis, we plot the cubic root of the number of degrees of freedom.
The stabilizations~$\widetilde S_1^\E$ and~$\SEz$ are defined in~\eqref{D-recipe:stabilizations} and~\eqref{explicit:stabilization:S0}, respectively.
The polynomial basis dual to the internal moments~\eqref{internal:moments} is $L^2$ orthonormal elementwise.
For the $\h$-version we employ uniform Cartesian meshes, for the $\p$-versions, we employ the meshes in Figure~\ref{figure:L-shaped}.
{ The error curves plotted in the figure refer to the
      approximation of the problem eigenvalues as follows:
      \textit{top-left panel:} first eigenvalue; 
      \textit{top-right panel:} second eigenvalue;
      \textit{bottom-left panel:} third eigenvalue;
      \textit{bottom-right panel:} fourth eigenvalue.}}
\label{figure:testcase3}
\end{figure}

From Figure~\ref{figure:testcase3}, it is possible to appreciate the exponential convergence of the $\h\p$-version of the method in terms of the cubic root of the number of degrees of freedom,
which indeed resembles its counterpart for the source problem, see~\cite[Theorem 3]{hpVEMcorner}.

\paragraph*{\texttt{Test case 4}: checkerboard diffusivity tensor}
As a final test, we consider as a physical domain the square $\Omega_4=(-1,1)^2$,
and we focus our attention to the checkerboard benchmark, that is, we fix in~\eqref{continuous:problem:strong} $\V=0$ and we define the diffusivity tensor as follows:
\begin{equation} \label{cb:diffusivity}
\K(x,y) = \varrho(x,y) \mathbb I \quad \text{where} \quad \varrho(x,y) =
\begin{cases}
\varepsilon & \text{in } (-1,0)^2 \cup (0,1)^2\\
1 & \text{elsewhere}.\\
\end{cases}
\end{equation}
The eigenfunctions could be singular at the center of the checkerboard, that is, at the origin of the axes.
Therefore, in addition to the $\h$-version of the method, we also consider the $\h\p$-version with a distribution of the ``polynomial'' degrees as in~\eqref{hp:degrees:of:accuracy} with $\mu=1$,
and geometrically graded meshes with grading parameter $\sigma=0.5$, as those depicted in Figure~\ref{figure:hp:checkerboard}.
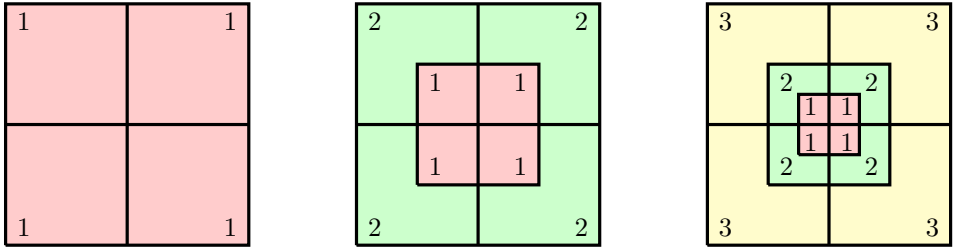
\begin{figure}  [h]
\begin{center}
\begin{minipage}{0.30\textwidth}
\begin{tikzpicture}[scale=0.4]
\fill[red, opacity=0.2]  (-4,-4) -- (4,-4) -- (4,4) -- (-4,4) -- (-4,-4);
\draw[black, very thick, -] (-4,-4) -- (4,-4) -- (4,4) -- (-4,4) -- (-4,-4);
\draw[black, very thick, -] (0,-4) -- (0,4); \draw[black, very thick, -] (-4,0) -- (4,0);
\draw(-3.4,-3.4) node[black] {{$1$}}; \draw(-3.4,3.4) node[black] {{$1$}}; \draw(3.4,-3.4) node[black] {{$1$}}; \draw(3.4,3.4) node[black] {{$1$}};
\end{tikzpicture}
\end{minipage}
\begin{minipage}{0.30\textwidth}
\begin{tikzpicture}[scale=0.4]
\fill[green, opacity=0.2]  (-4,-4) -- (4,-4) -- (4,4) -- (-4,4) -- (-4,-4);
\fill[white, opacity=1]  (-2,-2) -- (2,-2) -- (2,2) -- (-2,2) -- (-2,-2);
\fill[red, opacity=0.2]  (-2,-2) -- (2,-2) -- (2,2) -- (-2,2) -- (-2,-2);
\draw[black, very thick, -] (-4,-4) -- (4,-4) -- (4,4) -- (-4,4) -- (-4,-4);
\draw[black, very thick, -] (-2,-2) -- (2,-2) -- (2,2) -- (-2,2) -- (-2,-2);
\draw[black, very thick, -] (0,-4) -- (0,4); \draw[black, very thick, -] (-4,0) -- (4,0);
\draw(-3.4,-3.4) node[black] {{$2$}}; \draw(-3.4,3.4) node[black] {{$2$}}; \draw(3.4,-3.4) node[black] {{$2$}}; \draw(3.4,3.4) node[black] {{$2$}};
\draw(-1.4,-1.4) node[black] {{$1$}}; \draw(-1.4,1.4) node[black] {{$1$}}; \draw(1.4,-1.4) node[black] {{$1$}}; \draw(1.4,1.4) node[black] {{$1$}};
\end{tikzpicture}
\end{minipage}
\begin{minipage}{0.33\textwidth}
\begin{tikzpicture}[scale=0.4]
\fill[yellow, opacity=0.2]  (-4,-4) -- (4,-4) -- (4,4) -- (-4,4) -- (-4,-4);
\fill[white, opacity=1]  (-2,-2) -- (2,-2) -- (2,2) -- (-2,2) -- (-2,-2);
\fill[green, opacity=0.2]  (-2,-2) -- (2,-2) -- (2,2) -- (-2,2) -- (-2,-2);
\fill[white, opacity=1]  (-1,-1) -- (1,-1) -- (1,1) -- (-1,1) -- (-1,-1);
\fill[red, opacity=0.2]  (-1,-1) -- (1,-1) -- (1,1) -- (-1,1) -- (-1,-1);
\draw[black, very thick, -] (-4,-4) -- (4,-4) -- (4,4) -- (-4,4) -- (-4,-4);
\draw[black, very thick, -] (-2,-2) -- (2,-2) -- (2,2) -- (-2,2) -- (-2,-2);
\draw[black, very thick, -] (-1,-1) -- (1,-1) -- (1,1) -- (-1,1) -- (-1,-1);
\draw[black, very thick, -] (0,-4) -- (0,4); \draw[black, very thick, -] (-4,0) -- (4,0);
\draw(-3.4,-3.4) node[black] {{$3$}}; \draw(-3.4,3.4) node[black] {{$3$}}; \draw(3.4,-3.4) node[black] {{$3$}}; \draw(3.4,3.4) node[black] {{$3$}};
\draw(-1.4,-1.4) node[black] {{$2$}}; \draw(-1.4,1.4) node[black] {{$2$}}; \draw(1.4,-1.4) node[black] {{$2$}}; \draw(1.4,1.4) node[black] {{$2$}};
\draw(-.6,-.6) node[black] {{$1$}}; \draw(-.6,.6) node[black] {{$1$}}; \draw(.6,-.6) node[black] {{$1$}}; \draw(.6,.6) node[black] {{$1$}};
\end{tikzpicture}
\end{minipage}
\end{center}
\caption{First three meshes~$\mathcal T_0$,~$\mathcal T_1$, and~$\mathcal T_2$, and the distribution of the ``polynomial'' degree for the approximation of eigenfunctions and eigenvalues of the checkerboard benchmark on the square domain~$\Omega_4$.
The $0$-th, the $1$-st, and the $2$-nd layers are highlighted in \red{red}, \green{green}, and \yellow{yellow} colours, respectively.
The layers are constructed refining only towards the center of the checkerboard.}
\label{figure:hp:checkerboard}
\end{figure}

It is worthwhile to underline that the regularity of the solution to the associated source problem~\eqref{weak:formulation:continuous:source} decreases as $\varepsilon \rightarrow 0$ in~\eqref{cb:diffusivity}.
In~\cite[Figure 1]{CiarletJamelotKpadonou}, such a regularity was pinpointed for some specific choices of~$\varepsilon$.

For this test case, we fix homogeneous Neumann boundary conditions. In order to test the method, we compare our discrete eigenvalues with those in~\cite{Dauge_benchmark}.
In Figures~\ref{figure:testcase4:part1} and~\ref{figure:testcase4:part2}, we show the performance of the $\h$- and of the $\h\p$-versions on a sequence of uniform Cartesian meshes (conforming with respect to the discontinuities of the diffusivity tensor)
and on the sequence of graded meshes in Figure~\ref{figure:hp:checkerboard}, for two different ``checkerboard'' parameters~$\varepsilon$ in~\eqref{cb:diffusivity}, respectively.
More precisely, we study the convergence to the first four distinct eigenvalues, with parameters~$\varepsilon=2$ and~$10^8$.
\begin{figure}  [h]
\centering
\includegraphics [angle=0, width=0.45\textwidth]{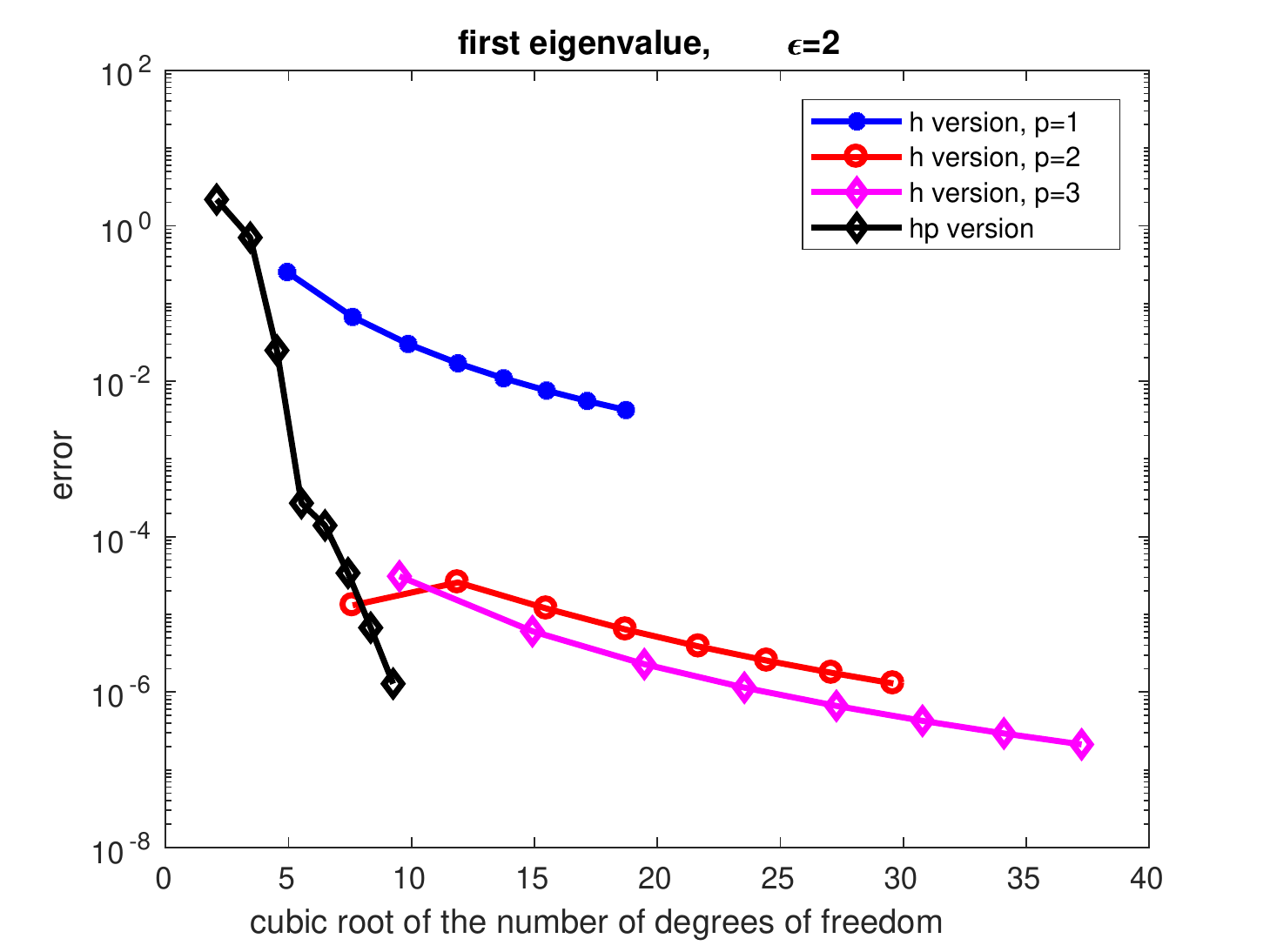}
\includegraphics [angle=0, width=0.45\textwidth]{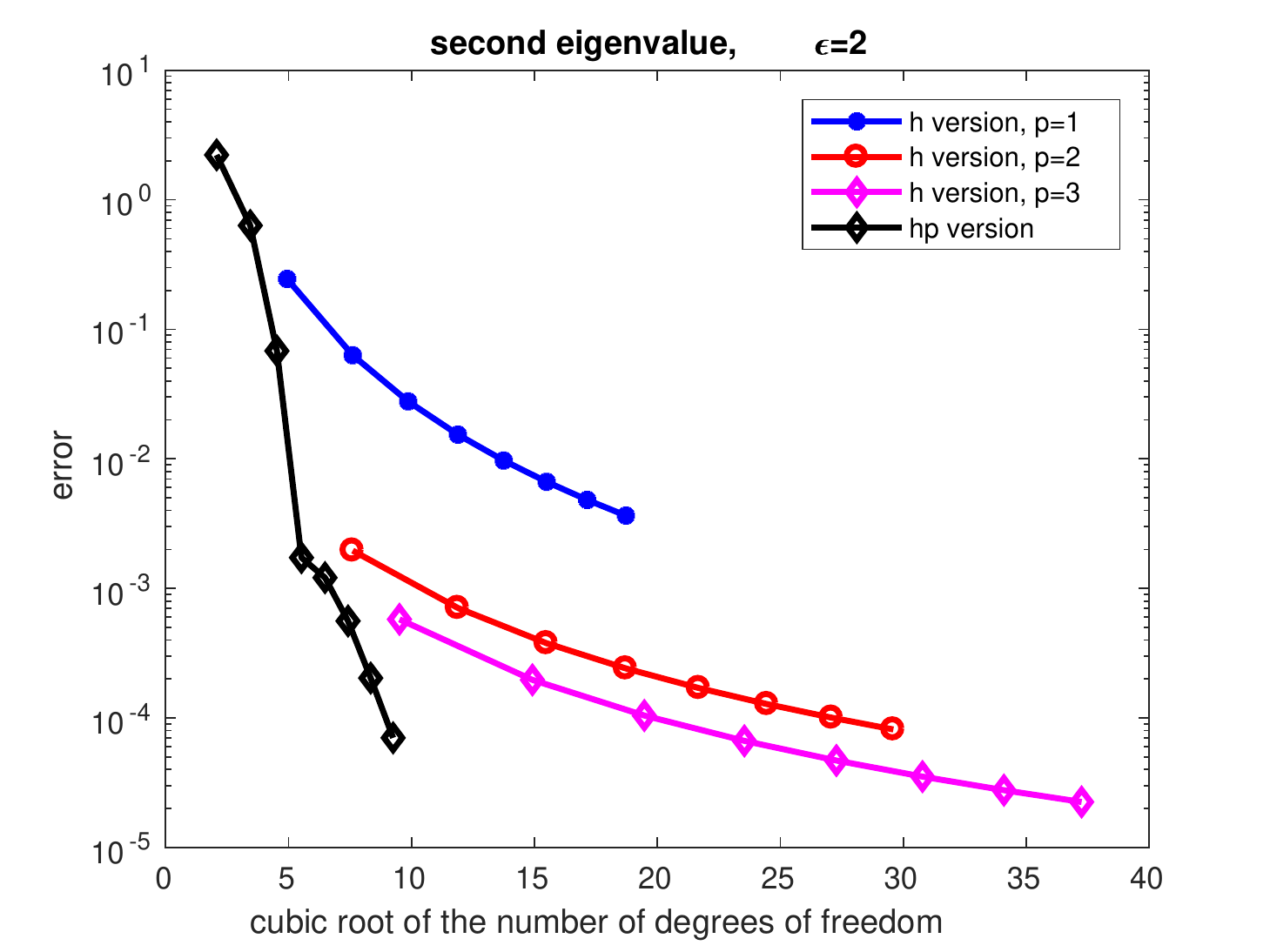}
\includegraphics [angle=0, width=0.45\textwidth]{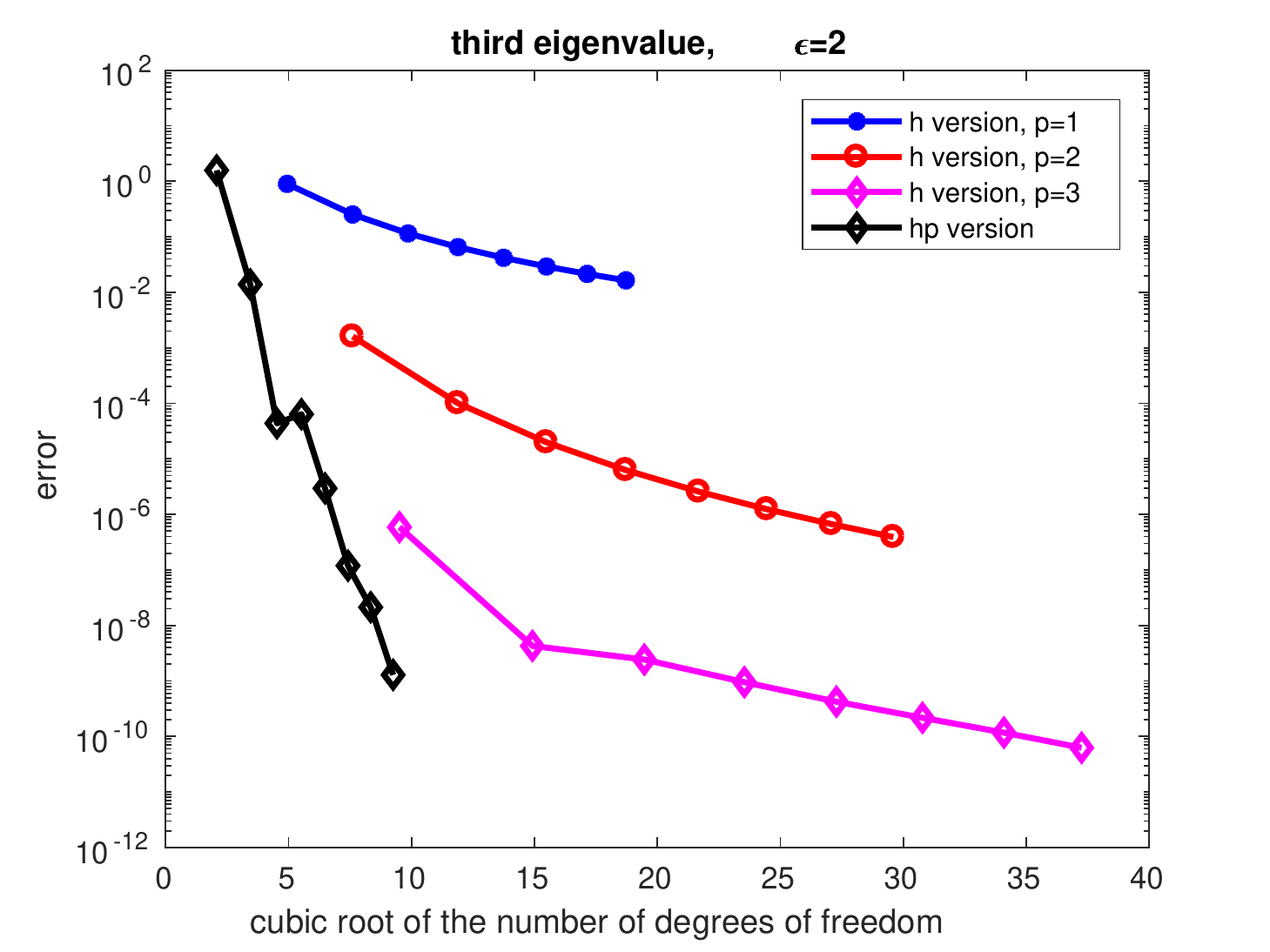}
\includegraphics [angle=0, width=0.45\textwidth]{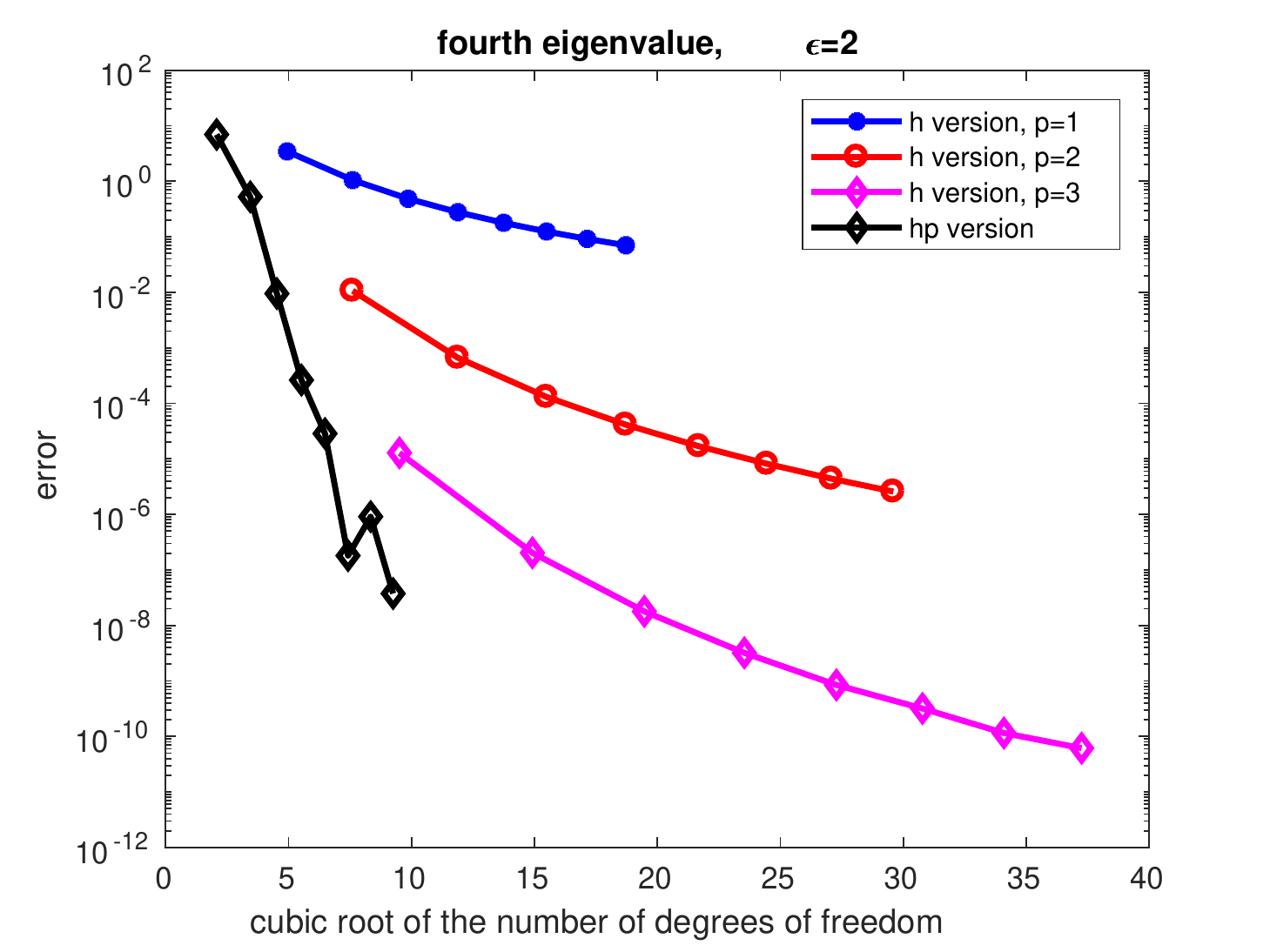}
\caption{Convergence of the error~\eqref{quantity:of:interest} for the first distinct Neumann four eigenvalues of the ``checkerboard'' operator with~$\varepsilon = 2$ on the square domain~$\Omega_4$
employing the $\h$-version with $\p=1$, $2$, and~$3$ and the $\h\p$-version of the method. On the $x$-axis, we plot the cubic root of the number of degrees of freedom.
The stabilizations~$\widetilde S_1^\E$ and~$\SEz$ are defined in~\eqref{D-recipe:stabilizations} and~\eqref{explicit:stabilization:S0}, respectively.
The polynomial basis dual to the internal moments~\eqref{internal:moments} is $L^2$ orthonormal elementwise.
For the $\h$-version we employ uniform Cartesian meshes, for the $\p$-versions, we employ the meshes in Figure~\ref{figure:L-shaped}.
The error curves plotted in the figure refer to the
      approximation of the problem eigenvalues as follows:
      \textit{top-left panel:} first eigenvalue; 
      \textit{top-right panel:} second eigenvalue;
      \textit{bottom-left panel:} third eigenvalue;
      \textit{bottom-right panel:} fourth eigenvalue}
\label{figure:testcase4:part1}
\end{figure}

\begin{figure}  [h]
\centering
\includegraphics [angle=0, width=0.45\textwidth]{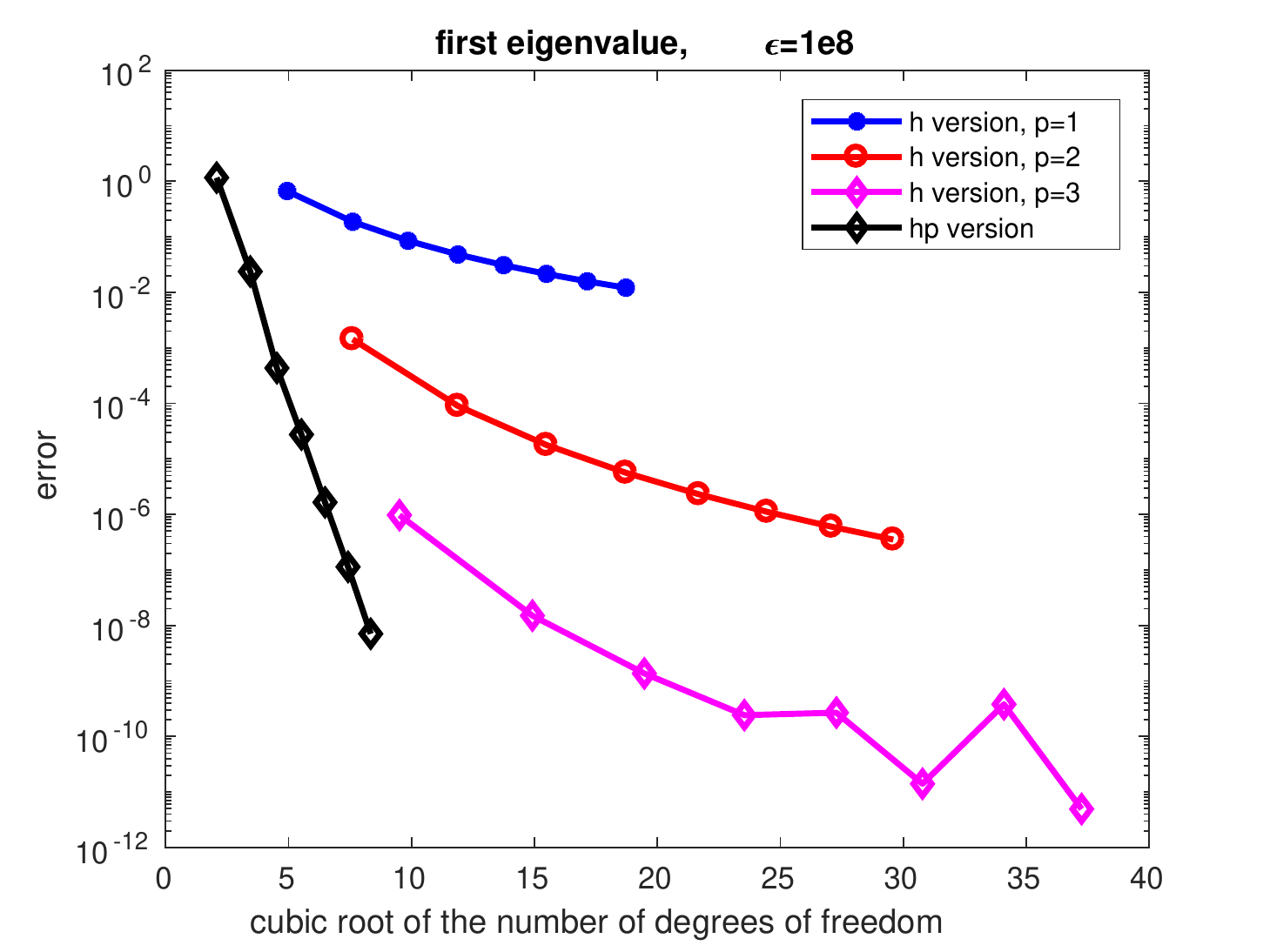}
\includegraphics [angle=0, width=0.45\textwidth]{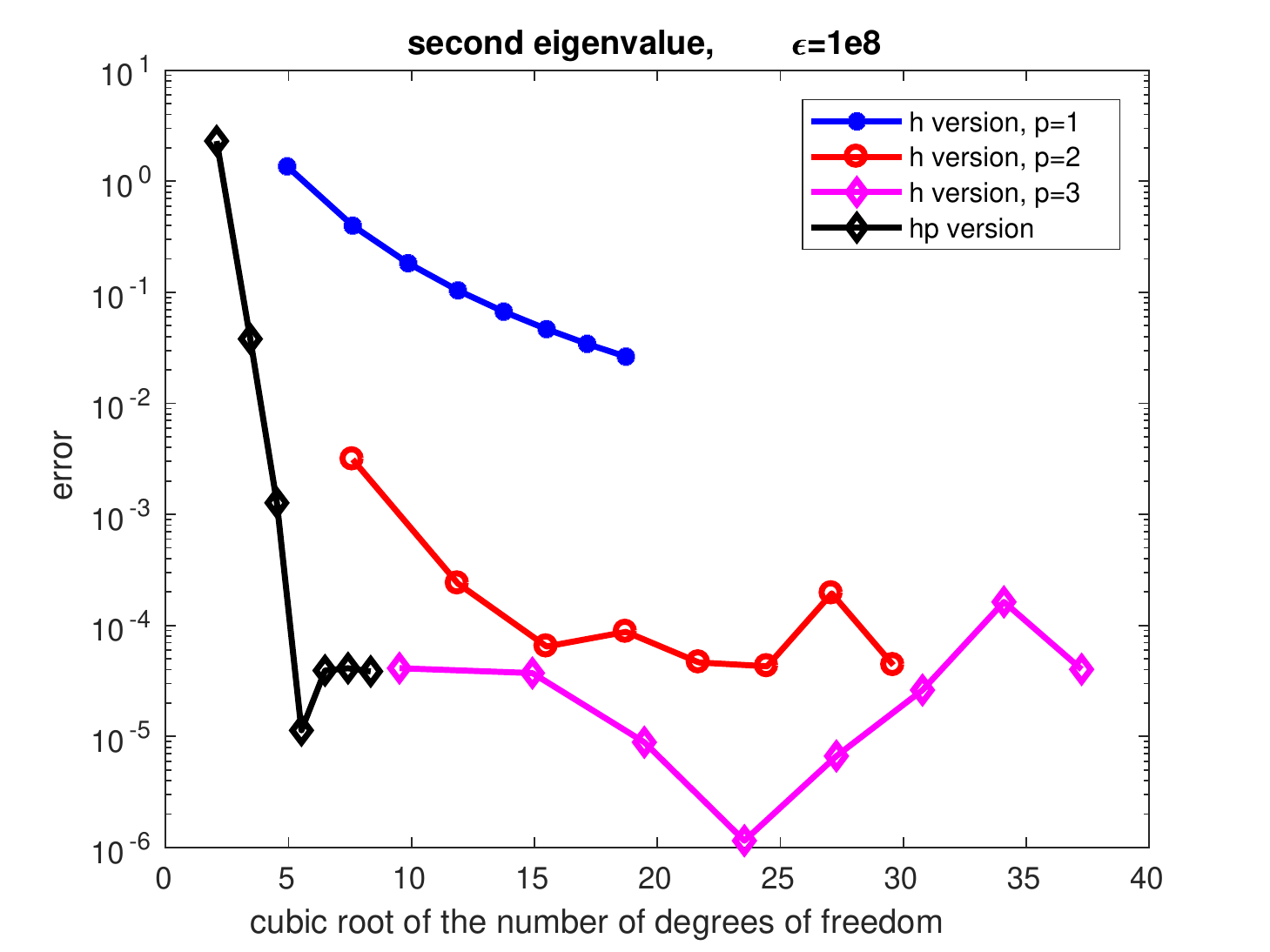}
\includegraphics [angle=0, width=0.45\textwidth]{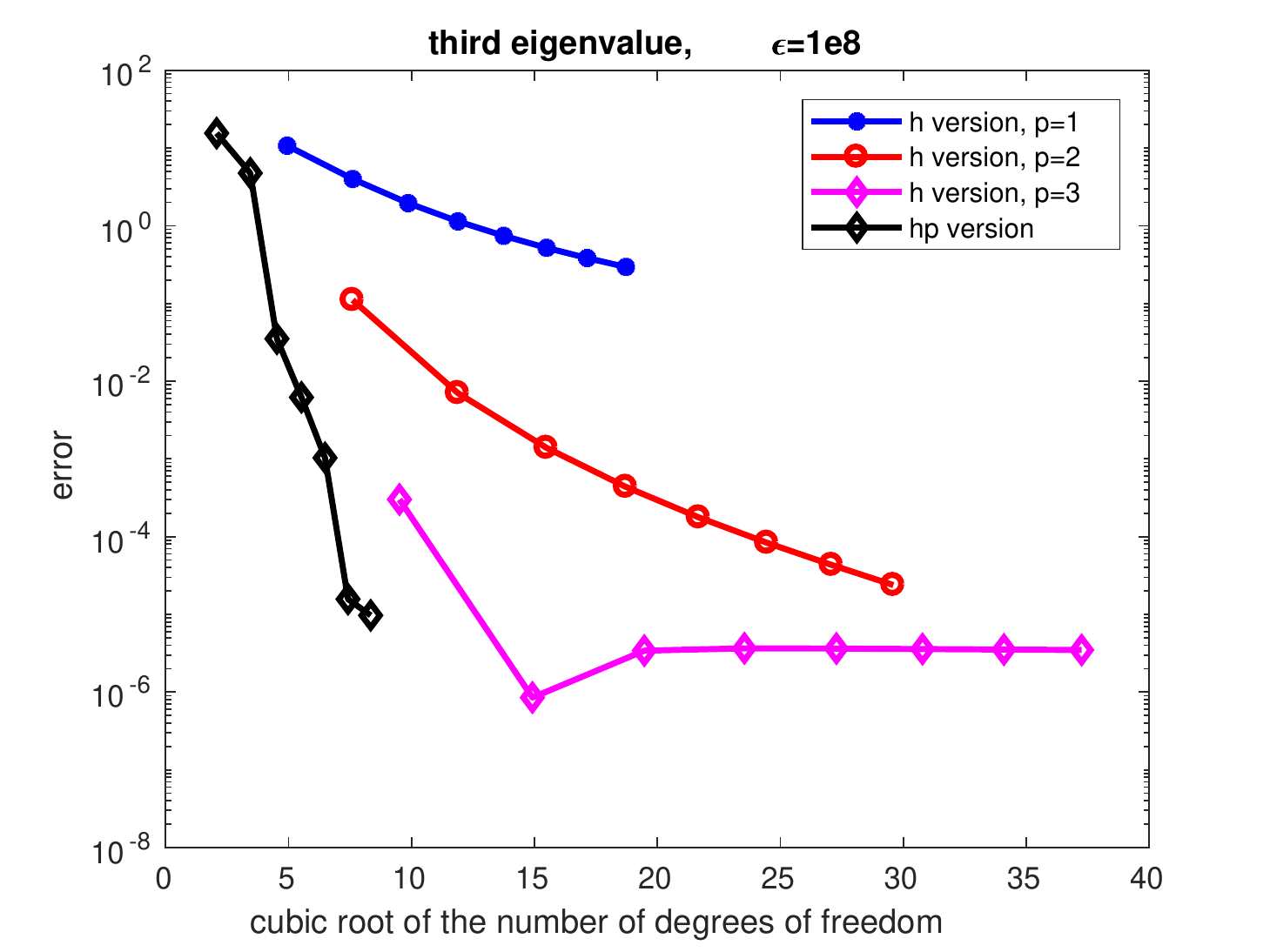}
\includegraphics [angle=0, width=0.45\textwidth]{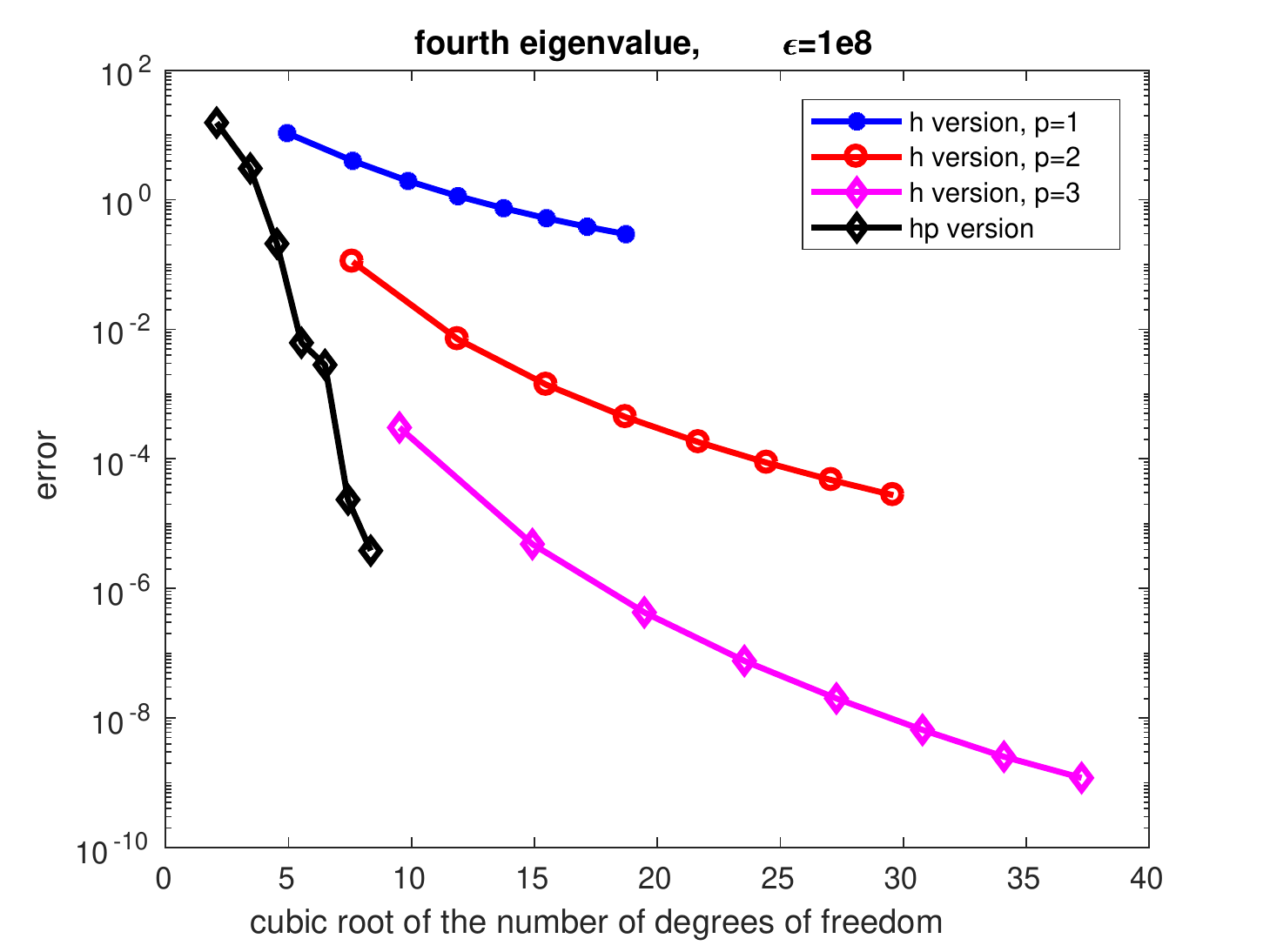}
\caption{Convergence of the error~\eqref{quantity:of:interest} for the first four distinct Neumann eigenvalues of the ``checkerboard'' operator with~$\varepsilon = 10^{8}$ on the square domain~$\Omega_4$
employing the $\h$-version with $\p=1$, $2$, and~$3$ and the $\h\p$-version of the method. On the $x$-axis, we plot the cubic root of the number of degrees of freedom.
The stabilizations~$\widetilde S_1^\E$ and~$\SEz$ are defined in~\eqref{D-recipe:stabilizations} and~\eqref{explicit:stabilization:S0}, respectively.
The polynomial basis dual to the internal moments~\eqref{internal:moments} is $L^2$ orthonormal elementwise.
For the $\h$-version we employ uniform Cartesian meshes, for the $\p$-versions, we employ the meshes in Figure~\ref{figure:L-shaped}.
The error curves plotted in the figure refer to the
      approximation of the problem eigenvalues as follows:
      \textit{top-left panel:} first eigenvalue; 
      \textit{top-right panel:} second eigenvalue;
      \textit{bottom-left panel:} third eigenvalue;
      \textit{bottom-right panel:} fourth eigenvalue.}
\label{figure:testcase4:part2}
\end{figure}

Again, the exponential convergence of the $\h\p$-version of the method in terms of the cubic root of the degrees of freedom is in accordance with~\cite[Theorem]{hpVEMcorner},
which is the analogous result for the source problem.
We underline that the poor convergence rate for the second and third eigenvalues when~$\varepsilon=10^8$ is due to the poor accuracy of the exact eigenvalues computed in~\cite{Dauge_benchmark}.
Importantly, the method is extremely robust for choices of both very high and moderate~$\varepsilon$.

%

\section{Conclusion} \label{section:conclusion}
We analyzed the $\p$-version of the virtual element method for elliptic problems with variable diffusivity tensor and reaction term.
Particular emphasis was stressed on $\p$-best interpolation estimates in enhanced virtual element spaces and on a careful investigation of the stabilization terms. Such analysis was instrumental to derive the a priori $\p$-convergence estimate for eigenvalue problems.
A wide set of numerical results, including experiments with the $\h\p$-version of the method, was presented in order to underline the superiority of the $\p$- and $\h\p$-versions of the method over the $\h$- one.

\paragraph*{Acknowledgements}
Lorenzo Mascotto has been funded by the Austrian Science Fund (FWF) through the project F~65.

The work of Ond\v{r}ej \v{C}ert\'ik and Gianmarco Manzini was supported by the Laboratory Directed Research and Development Program (LDRD), U.S. Department of Energy Office of Science, Office of Fusion Energy Sciences,
and the DOE Office of Science Advanced Scientific Computing Research (ASCR) Program in Applied Mathematics Research, under the auspices of the National Nuclear Security Administration of the U.S. Department of Energy by Los Alamos National Laboratory,
operated by Los Alamos National Security LLC under contract DE-AC52-06NA25396.
This work has been assigned the number LA-UR-18-31762.

The authors are grateful to Dr. Joscha Gedicke from the University of Vienna for useful discussions regarding various aspects of the approximation of eigenvalues by means of Galerkin methods.


{\footnotesize
\bibliography{bibliogr}
}
\bibliographystyle{plain}

\end{document}